\documentclass[12pt]{article}

\usepackage{amsmath}
\usepackage{latexsym}
\usepackage{amssymb}
\usepackage{ifthen}
\usepackage{tikz}
\usepackage{amsthm}
\usepackage{yfonts}
\usepackage{hyperref}

\newcommand{\junk}[1]{}
\newcommand{\formal}[1]{\ensuremath{\textsf{#1}}}
\newcommand{\dword}[1]{\textit{#1}}
\newcommand{\clnb}[1]{\ensuremath{\formal{N}[#1]}}
\newcommand{\graph}[1]{\ensuremath{\mathit{#1}}}
\newcommand{\capt}{\ensuremath{\mathrm{capt}}}

\newcommand{\rfunct}[1]{\ensuremath{\mathrm{cr}(#1)}}

\newcommand{\vect}[1]{\ensuremath{\mathbf{#1}}}

\newcommand{\uprank}[2]{\ensuremath{#1^{[#2 ]}}}
\newcommand{\X}{\ensuremath{U}}

\newcommand{\upcorner}{\emph{Upward Cornering}}

\newcommand{\tp}[1]{\ensuremath{#1}-top}

\newcommand{\pathcon}{\emph{Path Contraction}}

\newcommand{\power}[1]{\ensuremath{\mathcal{P}(#1)}}

\newtheorem{defn}{Definition}[section]
\newtheorem{theorem}[defn]{Theorem}
\newtheorem{proposition}[defn]{Proposition}
\newtheorem{cor}[defn]{Corollary}
\newtheorem{question}[defn]{Question}

\newtheorem{lemma}[defn]{Lemma}

\newtheorem{rem}[defn]{Remark}

\newtheorem*{def_nonum}{Definition}
\newtheorem*{lem_nonum}{Lemma}

\title{Capture-time Extremal Cop-Win Graphs}

\author{David Offner\\
\small Department of Mathematics and Computer Science\\[-0.8ex]
\small Westminster College\\[-0.8ex] 
\small New Wilmington, PA, U.S.A.\\
\small\tt offnerde@westminster.edu\\
\and
Kerry Ojakian\\
\small Department of Mathematics and Computer Science\\[-0.8ex]
\small Bronx Community College (CUNY)\\[-0.8ex]
\small Bronx, NY, U.S.A\\
\small\tt kerry.ojakian@bcc.cuny.edu
}

\begin{document}
\maketitle

\begin{abstract}  

We investigate extremal graphs related to the game of Cops and Robbers.  We focus on graphs where a single cop can catch the robber; such graphs are called cop-win. The capture time of a cop-win graph is the minimum number of moves the cop needs to capture the robber.  We consider graphs that are extremal with respect to capture time, i.e. their capture time is as large as possible given their order.  We give a new characterization of the set of extremal graphs.
For our alternative approach we assign a rank to each vertex of a graph, and then study which configurations of ranks are possible.  We partially determine which configurations are possible, enough to prove some further extremal results.  We leave a full classification as an open question.

\end{abstract}

\section{Introduction}

The game of Cops and Robbers is a perfect-information two-player pursuit-evasion game played on a  graph.  To begin the game, the cop and robber each choose a vertex to occupy, with the cop choosing first.  Play then alternates between the cop and the robber, with the cop moving first. On a turn a player may move to an adjacent vertex or stay still.  If the cop and robber ever occupy the same vertex, the robber is caught and the cop wins.  If the cop can force a win on a graph, we say the graph is \dword{cop-win}. The game was introduced by Nowakowski and Winkler \cite{NW83}, and Quilliot \cite{Qui78}. A nice introduction to the game and its many variants is found in 
the book by Bonato and Nowakowski \cite{BN11}. 

One of the fundamental results about the game is a characterization of the cop-win graphs as those graphs which have a \dword{cop-win ordering} \cite{NW83}, \cite{Qui78}.   
Independently, Clarke, Finbow, and MacGillivray \cite{CFM2014} and the authors of this paper~\cite{OO16} developed an alternative characterization that we call \emph{corner ranking}. A thorough discussion of the similarities and differences of our approach is given in \cite{OO16}. As with cop-win orderings, corner ranking characterizes which graphs are cop-win.  Corner ranking can also be used to determine the capture time of a cop-win graph $G$, as well as describe optimal strategies (in terms of capture time) for the cop and robber, where the \dword{capture time} of a cop-win graph $G$, denoted $\capt(G)$, is the fewest number of moves the cop needs to guarantee a win, not counting their initial placement. For example, on the path with $5$ vertices, the capture time is 
$2$.  In Section~\ref{cornran}, we summarize some work
from \cite{OO16} on corner ranking.

Bonato et. al.
 \cite{BGHK2009} defined the following capture time function on the natural numbers.
\begin{defn} \label{def_capt}
Suppose $n > 0$ is a natural number. Let 
$\capt(n) $ denote the capture time of a cop-win graph on $n$ vertices with maximum capture time.
\end{defn}

\noindent
For example, $\capt(4) = 2$ since a path on four vertices has capture time 2, and no graph with 4 vertices has a capture time greater than 2. 
Define a cop-win graph $G$ with $n$ vertices to be \dword{CT-maximal} if $\capt(G) = \capt(n)$.
Building on \cite{BGHK2009}, Gavenciak~\cite{Gav2010} proved that for $n\ge 7$,  $\capt(n) = n - 4$, and gave a characterization of the CT-maximal graphs. More recently, Kinnersley~\cite{Kin17} has studied upper bounds on capture time of graphs where more than one cop is required to catch the robber.
Gavenciak's proof relies on a detailed analysis of the conceivable cop and robber strategies.  We give an alternative proof (in Theorem~\ref{gave}), which instead proceeds by analyzing the structure of graphs using a tool we call the \emph{rank cardinality list}. One advantage of our approach is
that it makes case analysis easier.
In fact, Gavenciak uses a computer at one step in the proof (Lemma 10 in \cite{Gav2010}) to analyze graphs of order less than or equal 
to 8.  We can carry out the analysis without a computer, using the theory we develop about rank cardinality lists.

Our approach to the proofs is to associate cop-win graphs with finite lists. The corner ranking procedure assigns each vertex in a cop-win graph an integer, so in Section~\ref{RC} we define the \emph{rank cardinality list} of a cop-win graph as the list whose $i$th entry is the number of vertices of corner rank $i$.  Since the length of the list is the corner rank of the graph, which determines capture time, 
we can characterize the CT-maximal graphs by determining which lists are \emph{realizable}, i.e. which lists are the rank cardinality list for some cop-win graph.  Thus the fundamental issue in our paper becomes determining which lists are realizable and which are not.

In Section~\ref{RC} we determine enough about the realizability of lists to prove Theorem~\ref{gave}, our first main theorem, which we can restate using the following definition.

\begin{defn} \label{def_graph_n_t}
Let $\mathcal{G}_n^s$ be the set of cop-win graphs with $n$ vertices and capture time $s$.
\end{defn} 

\noindent
Theorem~\ref{gave} gives a characterization of 
$\mathcal{G}_n^{n-4}$.  In Section~\ref{morrc} we determine more about the realizability of lists, enough to prove Theorem~\ref{thm_n-5}, our second main theorem, which provides a characterization of $\mathcal{G}_n^{n-5}$.  To prove our two main theorems we partially characterize the realizable lists.  In Section~\ref{future} we suggest fully characterizing the realizable lists as a direction for future work, and mention some preliminary results.

\section{Corner Ranking}\label{cornran}

In this section we review the necessary results about corner rank from our paper \cite{OO16}. Since \cite{OO16} is currently unpublished, some proofs from that paper are included, albeit in a concise manner, in the appendix. 
For a full development, including proofs and examples, see \cite{OO16}. In this paper all numbers are integers, and all graphs are finite and non-empty, i.e. they have at least one vertex.
We follow a typical Cops and Robbers convention by assuming that all graphs are reflexive, that is all graphs have a loop at every vertex so that a vertex is always adjacent to itself
(In figures we never draw such edges).  
For a graph \graph{G}, $V(\graph{G})$ refers to the vertices of \graph{G} and $E(\graph{G})$ refers to the edges of \graph{G}.
If \graph{G} is a graph and $X$ is a vertex or set of vertices in \graph{G},
then by $\graph{G} - X$ we mean the subgraph of $\graph{G}$ induced by $V(\graph{G})\setminus X$.  Given a vertex $v$ in a graph, by the \dword{closed neighborhood of $v$}, denoted \clnb{v}, we mean the set of vertices consisting of $v$ and all the vertices adjacent to $v$. We say that a vertex $v$ \dword{dominates} a set of vertices $X$ if $X \subseteq \clnb{v}$.
For distinct vertices $v$ and $w$, if $\clnb{v} \subseteq \clnb{w}$ then
we say that $v$ is a \dword{corner} and that $w$ \dword{corners} $v$;  
if $\clnb{v} \subsetneq \clnb{w}$, we say that $v$ is a \dword{strict corner} and that $w$ \dword{strictly corners} $v$.
If $\clnb{v} = \clnb{w}$, we call $v$ and $w$ \dword{twins}.

A \dword{cop-win ordering} of a graph (also called a \dword{dismantling ordering}) \cite{NW83, Qui78} is produced by removing one corner at a time, until all the vertices have been removed (note that only cop-win graphs have cop-win orderings).  
As a small but significant modification of the cop-win ordering, rather than removing one corner at a time, we remove all the current strict corners simultaneously, assigning them a number we call the \emph{corner rank}.
In this paper, we only apply corner ranking to cop-win graphs,
though a more general approach is described in \cite{OO16}.

\begin{defn}{\bf (Corner Ranking Procedure)}
\label{def_corner_ranking}
For any cop-win graph \graph{G},
we define a corresponding  \dword{corner rank} function, $\mathrm{cr}$, which maps each vertex of $\graph{G}$ to a positive integer.
We also define a sequence of associated graphs $\uprank{\graph{G}}{1}, \ldots, \uprank{\graph{G}}{\alpha}$.

\begin{enumerate}
\setcounter{enumi}{-1}

\item Initialize $\uprank{\graph{G}}{1}=\graph{G}$, and $k = 1$.

\item 
\label{alg_start}
If $\uprank{\graph{G}}{k}$ is a clique, then:
\begin{itemize}

\item
 Let $\rfunct{x} = k$ for all $x \in V(\uprank{\graph{G}}{k})$.

\item
Then stop.

\end{itemize}


\item Else:  
\begin{itemize}
\item Let \X \ be the set of strict corners in \uprank{\graph{G}}{k}.
\item For all $x \in \X$, let  $\rfunct{x} = k$.  
\item Let $\uprank{\graph{G}}{k+1} = \uprank{\graph{G}}{k} - \X$. 
\item Increment $k$ by 1 and return to Step~\ref{alg_start}.
\end{itemize}
\end{enumerate}
Define the \dword{corner rank} of \graph{G}, denoted \rfunct{\graph{G}}, to be the same as a vertex of \graph{G}
with largest corner rank.

\end{defn}

The corner ranking procedure is well-defined, giving a corner rank to every vertex in a cop-win graph (See the appendix for a proof). As an example, we apply the corner ranking procedure to the graph $\graph{H}_7$, which is drawn in
two different ways in Figure~\ref{h7}. This graph was introduced in \cite{BGHK2009} and is typically drawn as the graph on the left in the figure. 
The corner ranking procedure begins by assigning the 
strict corner $d$ rank 1. After $d$ is removed, $c_1$ and $c_2$ are strict corners, 
and are thus assigned corner rank 2. Likewise, $b_1$ and $b_2$ are assigned corner rank 3.
After $b_1$ and $b_2$ are removed, the remaining vertices, $a_1$ and $a_2$,
form a clique and so are assigned corner rank 4. Thus the corner rank of the graph $\graph{H}_7$ is 4. The graph drawn on the right in Figure~\ref{h7} shows the graph $\graph{H}_7$ with
its corner rank structure more clearly displayed.

\begin{rem}
 In all figures, when a vertex $w$ has rank $k$ and is strictly cornered in \uprank{\graph{G}}{k} by a vertex $v$ of higher rank, we draw the edge $vw$ with a thick line.
Also, the number drawn inside a vertex indicates its corner rank. 
\end{rem}

\begin{figure}
\begin{center}
  \begin{tikzpicture}[every node/.style={circle,fill=black!20}, scale=.833]  
      \node (a2) at (2,0) [label=above left:$a_1$] {$4$};
      \node (a1) at (4,0) [label=above right:$a_2$] {$4$};
      \node (c1) at (0,0) [label=left:$c_2$] {$2$};
      \node (c2) at (6,0) [label=right:$c_1$] {$2$};
      \node (b1) at (3,-1.5) [label=left:$b_1$] {$3$};
      \node (b2) at (3,2) [label=right:$b_2$]  {$3$};
      \node (d) at (3,-3) [label=right:$d$]  {$1$};
    \foreach \from/\to in {a1/a2,  b1/c1, b1/c2, b2/c1, b2/c2, c1/d, c2/d}
    \draw (\from) -- (\to);
    \foreach \from/\to in {a1/b2, a1/b1, a2/b2, a2/b1, a2/c1, a1/c2, b1/d}
    \draw[ultra thick] (\from) -- (\to);
    \end{tikzpicture}
\hspace{.5in}
  \begin{tikzpicture}[every node/.style={circle,fill=black!20}, scale=1]  
      \node (a1) at (0,6) [label=left:$a_1$] {$4$};
      \node (a2) at (2,6) [label=right:$a_2$] {$4$};
      \node (b1) at (0,4.5) [label=left:$b_1$] {$3$};
      \node (b2) at (2,4.5) [label=right:$b_2$] {$3$};
      \node (c1) at (0,3) [label=left:$c_1$] {$2$};
      \node (c2) at (2,3) [label=right:$c_2$]  {$2$};
      \node (d) at (1,2) [label=right:$d$]  {$1$};
    \foreach \from/\to in {a1/a2,  b1/c1, b1/c2, b2/c1, b2/c2, c1/d, c2/d}
    \draw (\from) -- (\to);
    \foreach \from/\to in {a1/b2, a1/b1, a2/b2, a2/b1, a2/c1, a1/c2, b1/d}
    \draw[ultra thick] (\from) -- (\to);
    \end{tikzpicture}
\end{center}
\caption{Two representations of the graph $\graph{H}_7$.}
\label{h7}
\end{figure}
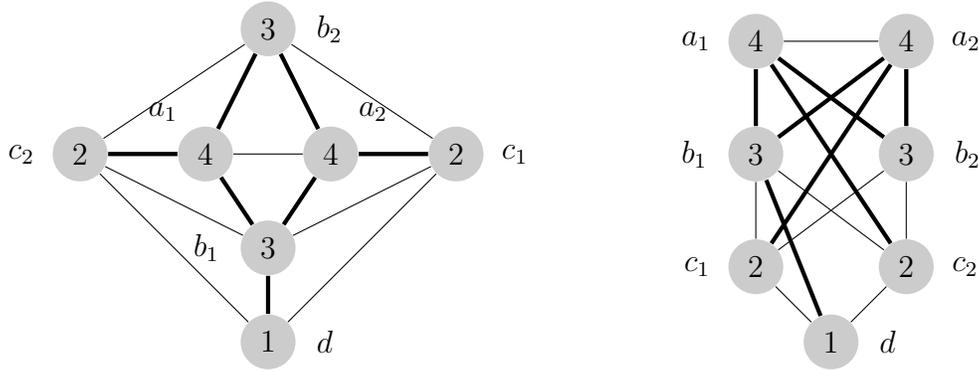

We define an important
structural property of the highest ranked vertices.

\begin{defn} \label{defn_topheavy}
Suppose \graph{G} is a cop-win graph with corner rank $\alpha$.
We say that \graph{G} is a \dword{\tp{1} graph} if \graph{G} has only one vertex, or if \graph{G} is a non-clique with some vertex of corner rank $\alpha$ dominating $V(\uprank{\graph{G}}{\alpha - 1})$. If \graph{G} is a clique with more than one vertex, or if \graph{G} is a non-clique with no vertex of corner rank $\alpha$ dominating $V(\uprank{\graph{G}}{\alpha - 1})$, we say \graph{G} is a \dword{\tp{0} graph}.
\end{defn}

\noindent
If \graph{G} has more than one vertex and 
is \tp{1}, then in fact every vertex of corner rank $\alpha$ dominates $V(\uprank{\graph{G}}{\alpha - 1})$
(a short proof is given in \cite{OO16}). 
For example, in Figure~\ref{2531} (ignore the caption for now), the graph on the left is \tp{1}, while the one on the right is \tp{0}. Note that in \cite{OO16} we used the name \emph{$t$-cop-win} in place of \tp{t}.

We now state a simplified version of the main result from \cite{OO16} (Theorem 6.1),
which relates the corner rank of a graph to its capture time. Since the theorem is proved in a currently unpublished manuscript, a succinct proof of the theorem is given in the
appendix. 

\begin{theorem} \label{thm_rank_capt_time}
For a \tp{t} cop-win graph \graph{G}, 
$\capt(\graph{G}) = \rfunct{\graph{G}} - t.$

\end{theorem}

\noindent
For example, the graph $\graph{H}_7$ in Figure~\ref{h7} is \tp{1} with corner rank 4, so it is cop-win with capture time $4-1=3$.

We give the following technical lemma the name
\upcorner, since we will use it so often. This lemma
follows immediately from Lemma 2.3 of \cite{OO16}, and
we also provide a proof in the appendix.

\begin{lemma}[\upcorner]
\label{higran}
If a vertex $v$ has corner rank $k$ in a graph $\graph{G}$ of rank larger than $k$, then $v$ is strictly cornered in \uprank{\graph{G}}{k} by a vertex of higher rank.

\end{lemma}

\section{Rank Cardinality Lists and Realizability}\label{RC}

Throughout this section, we assume that all graphs are cop-win, with corner rank at least 2.
For $1\le k \le \alpha$,  let $\X_k(\graph{G})$ denote the set of vertices of rank $k$ in a graph \graph{G}. We just write $\X_k$ if the graph is apparent from context.
By the term \dword{list}, we mean a finite list of positive integers.
A list $\vect{x} = (x_{\alpha}, x_{\alpha-1}, \ldots, x_1)$ has \dword{length} $\alpha$ and \dword{sum} $(x_\alpha + \cdots + x_1)$. Note that in our lists the indices decrease from $\alpha$ to 1, and for any number $x$, when we write $x, \ldots, x$ we mean a list of some number of $x$'s (at least one).

\begin{defn}
The  \dword{rank cardinality list} of a graph \graph{G} is the list $(x_{\alpha}, x_{\alpha-1}, \ldots, x_1)$, where for $k = 1, \ldots, \alpha$, $x_k$ is the number of vertices of rank $k$, i.e. $x_k = | \X_k |$.
\end{defn}

\begin{defn}
A list  $\vect{x} = (x_{\alpha}, x_{\alpha-1}, \ldots, x_1)$ is \dword{realizable} if it is the rank cardinality list of some cop-win graph \graph{G}.  
We say that \graph{G} \dword{realizes} \vect{x}, or that \vect{x} is \dword{realized by} \graph{G}.
For $t\in\{0,1\}$, \vect{x} is \dword{$t$-realizable} if there is a \tp{t} graph \graph{G} that realizes it.
We say that \graph{G} \dword{$t$-realizes} \vect{x}, or that \vect{x} is \dword{$t$-realized by} \graph{G}.

\end{defn}

\noindent
For example, the graph $\graph{H}_7$ in Figure~\ref{h7} realizes $(2,2,2,1)$, so since $\graph{H}_7$ is a \tp{1} graph, the list (2,2,2,1) is $1$-realizable.   We will see that some lists are not realizable. Since a $t$-realizable list with sum $n$ and length $\alpha$ corresponds to a \tp{t} graph on $n$ vertices with capture time $\alpha-t$, the answer to the following question would allow us to characterize $\mathcal{G}_n^{\alpha - t}$ and to determine $\capt(n)$.

\begin{question}\label{mainq}

For $t\in\{0,1\}$, which lists are $t$-realizable?

\end{question}

In this section, we answer this question to the extent necessary to give a proof of Theorem~\ref{gave}.  In Sections~\ref{morrc} and \ref{future}, we further explore this question and the general issue of realizability.

\subsection{Augmentations, Initial Segments, and Extensions}
We introduce three ways to alter a realizable list to obtain another realizable list: taking an augmentation, initial segment, or standard extension.

\begin{defn} \label{def_vecop}
Consider a list $(x_{\alpha}, \ldots, x_1)$.  
\begin{itemize}
\item If the list $(y_{\alpha}, \ldots, y_1)$ has the property that $x_i \le y_i$ for all $1\le i \le \alpha$, we say that $(y_{\alpha}, \ldots, y_1)$ is an \dword{augmentation} of $(x_{\alpha}, \ldots, x_1)$.

\item For $k \ge 1$, any list of the form $(x_{\alpha}, \ldots, x_k)$ is called an \dword{initial segment} of $(x_{\alpha}, \ldots, x_1)$.

\item Any list of the form $(x_{\alpha}, \ldots, x_1, z_1, z_2, \ldots, z_l)$ is called an \dword{extension} of $(x_{\alpha}, \ldots, x_1)$.  If $z_i=x_1$ for all $1 \le i \le l$, it is called a \dword{standard extension}.

\item
For all the notions (augmentation, initial segment, extension, and standard extension), we include the trivial case in which the list is unchanged.

\item  We say that  $\vect{x} \le \vect{y}$ if $\vect{y}$ is an augmentation (possibly trivial) of a standard extension (possibly trivial) of $\vect{x}$.

\end{itemize}
\end{defn}

\noindent
For example, a standard extension of $(3,2,2)$ is $(3,2,2,2,2)$ and an augmentation of  $(3,2,2,2,2)$ is $(5,2,6,2,3)$, so
$(3,2,2) \le (5,2,6,2,3)$.

\begin{proposition}\label{pump}
If a list 
is $t$-realizable, then so is any augmentation of it. 
\end{proposition}
\begin{proof}
It suffices to show that if  $\vect{x} =  (x_\alpha, \ldots, x_1)$ is $t$-realizable, then so is $\vect{y} = (y_\alpha, \ldots, y_1)$, where for some $k$, $y_k = x_k+1$, and for $j \neq k$, $y_j = x_j$.
Consider a graph \graph{G} which $t$-realizes $\vect{x}$. Choose a vertex $v \in \X_k$, and let $\graph{G}'$ be the graph obtained by adding a twin of $v$  to $\graph{G}$. Then $\graph{G}'$ $t$-realizes the list $\vect{y}$.
\end{proof}

If \graph{G} $t$-realizes the list $(x_\alpha, \ldots, x_1)$, then the initial segment $(x_\alpha, \ldots, x_k)$ is realized by \uprank{\graph{G}}{k}. Thus we obtain the following proposition.

\begin{proposition}\label{trunc}
If a list is $t$-realizable, then so is any initial segment of length 2 or more.
\end{proposition}

\begin{proposition}\label{extend}
Suppose $\vect{x} =  (x_\alpha, \ldots, x_1)$ and $\vect{y} =  (x_\alpha, \ldots, x_1, y_k, \ldots, y_1)$ is a standard extension. If \vect{x} is $t$-realizable then so is \vect{y}.
Moreover, if \graph{H} realizes \vect{x}, then there is a graph \graph{G} realizing \vect{y} such that $\uprank{\graph{G}}{k+1} = \graph{H}$.
\end{proposition}

\begin{proof}
It suffices to show that if  $\vect{x} =  (x_\alpha, \ldots, x_1)$ is $t$-realized by $\graph{H}$, then  $(x_\alpha, \ldots, x_1, x_1)$ is $t$-realized by some $\graph{G}$ where $\uprank{\graph{G}}{2}= \graph{H}$.  Suppose $\graph{H}$ $t$-realizes $\vect{x}$ with rank 1 vertices $v_1, \ldots, v_{x_1}$.  Let $\graph{G}$ be the graph obtained by adding the following to \graph{H}:
 vertices $w_1, \ldots, w_{x_1}$ and edges $v_1w_1,\ldots, v_{x_1}w_{x_1}$.  Then the vertices  $w_1, \ldots, w_{x_1}$ are the only strict corners in $\graph{G}$, the rank cardinality list of $\graph{G}$ is $(x_\alpha, \ldots, x_1, x_1)$, and $\uprank{\graph{G}}{2} = \graph{H}$.
\end{proof}

\noindent
From Propositions~\ref{pump} and \ref{extend}, we conclude the following.

\begin{cor}\label{augst}
For two lists $\vect{x}$ and $\vect{y}$ where $\vect{x} \le \vect{y}$, if \vect{x} is $t$-realizable, then \vect{y} is $t$-realizable.
\end{cor}

\noindent
As a special case, note that if $\vect{x} = (x_\alpha, \ldots, x_1)$ is $t$-realizable and $x_1 = 1$, then any extension of \vect{x} is $t$-realizable. We will often use the contrapositive form of Corollary~\ref{augst}: If $\vect{x} \le \vect{y}$, and \vect{y} is not $t$-realizable, then \vect{x} is not $t$-realizable. For example, in Corollary~\ref{cor_not12k1_not13k1} we show that for any $k$, the list $(1,3,k,1)$ is not realizable, which also implies that any list of the form $(1,2,k,1)$ is not realizable.

\subsection{Lists: Realizable and Not Realizable}

The main goal of this subsection is to prove a number of results about realizability: 1) some results show that a particular kind of list is realizable, 2) some results show that a particular kind of list is not realizable, and 
3) some results place restrictions on the structure of graphs realizing particular lists.  We begin with some technical results.

If the vertices of a cop-win graph are listed, beginning with all the corner rank 1 vertices, followed by all the corner rank 2 vertices, and so on, we arrive at a cop-win ordering
(repeatedly applying \upcorner \  shows that this is a cop-win ordering, though this fact is also proven as Lemma 6.5 of \cite{OO16}).  In a cop-win ordering, we can view each vertex as being retracted to a vertex later in the sequence which corners it.  It is well-known that this retract is isometric, thus the following proposition follows immediately (we name the proposition \pathcon, since we will want to refer to it often).

\begin{proposition}[\pathcon] 
\label{lonhin}
If $v$ and $w$ are vertices in \graph{G} of rank $k$ where the shortest path from $v$ to $w$ in \uprank{\graph{G}}{k} has length $m$, then there is no path from $v$ to $w$ in \graph{G} of length less than $m$.
\end{proposition}

\noindent
Proposition~\ref{lonhin} will be used as a tool to show many configurations are impossible.  For example, if $v$ and $w$ are nonadjacent vertices of rank $k$ without a common neighbor of rank $k$ or higher, they cannot have a common neighbor at all.


\begin{proposition}\label{nbrrank}
Suppose $v$ is a vertex of rank $k>1$.  Then for every vertex $w$ that strictly corners $v$ in \uprank{\graph{G}}{k}, $v$ must have a neighbor of rank $k-1$ that is not adjacent to $w$.
\end{proposition}

\begin{proof}
If not, then there is a vertex $w$ that strictly corners $v$ in \uprank{\graph{G}}{k-1}, contradicting the assumption that $v$ has rank $k$.
\end{proof}

\begin{cor}\label{xk1}
In a graph with rank $\alpha$, every vertex of rank $k>1$ has at least one neighbor of rank $k-1$. 
In particular, if there is  exactly one vertex $v$ of rank $k$, for some $k < \alpha$, then $v$ is adjacent to all the vertices of rank $k+1$.
\end{cor}

\begin{proposition}\label{nox2d}
In a graph with rank $\alpha$, no vertex of rank $\alpha-1$ can dominate $\X_{\alpha-1}$.
\end{proposition}

\begin{proof}
Suppose some vertex $b$  of rank $\alpha-1$ dominates $\X_{\alpha-1}$.
By \upcorner, let $a$ be a vertex of rank $\alpha$ that strictly corners $b$ in $\uprank{\graph{G}}{\alpha-1}$. Then $a$ must also dominate $\X_{\alpha-1}$, making $\graph{G}$ \tp{1}. In a \tp{1} graph, every vertex of rank $\alpha$ dominates $\X_{\alpha-1}$, and so $b$ is adjacent to every vertex of rank $\alpha$.  Thus $b$ is adjacent to every vertex in \uprank{\graph{G}}{\alpha-1}, contradicting the assumption that $a$ strictly corners $b$ in 
\uprank{\graph{G}}{\alpha-1}.
\end{proof}

\begin{cor}\label{norca-1}
No list $(x_\alpha, \ldots, x_1)$ with $x_{\alpha-1} = 1$ is realizable.
\end{cor}

\begin{proposition}\label{norca-2}
No list $(x_\alpha, \ldots, x_1)$ with $x_{\alpha-2} = 1$ is realizable.
\end{proposition}
\begin{proof}
Suppose $\graph{G}$ is a graph realizing $(x_{\alpha}, x_{\alpha-1}, \ldots, x_1)$, where $x_{\alpha-2} = 1$, and $c$ is the unique vertex of rank $\alpha-2$. By Corollary~\ref{xk1}, $\X_{\alpha-1} \subseteq \clnb{c}$. 
By \upcorner, some vertex $x$ of rank at least $\alpha - 1$ strictly corners $c$ in \uprank{\graph{G}}{\alpha - 2}.
If $x \in \X_{\alpha-1}$, then $x$ dominates $\X_{\alpha-1}$, which contradicts Proposition~\ref{nox2d}. If $x \in \X_{\alpha}$, then $x$ is adjacent to every vertex in \uprank{\graph{G}}{\alpha-2} and either strictly corners or is a twin of every other vertex. Thus \uprank{\graph{G}}{\alpha - 2} has rank at most 2, which contradicts the assumption that \uprank{\graph{G}}{\alpha - 2} has rank 3. 
\end{proof}

While the set of realizable lists includes lists that are not $0$-realizable, the set of realizable lists is in fact the same as the set of $1$-realizable lists.

\begin{proposition}\label{all1real}
Every realizable list of length 2 or more is $1$-realizable.
\end{proposition}

\begin{proof}
Suppose $\vect{x} = (x_\alpha, \ldots, x_1)$ is a realizable list, realized by \graph{G}.  If $x_\alpha=1$,  then \graph{G} must be \tp{1} so 
\vect{x} is $1$-realizable (though not $0$-realizable). Suppose $x_\alpha>1$.  By Corollary~\ref{norca-1} and Proposition~\ref{norca-2}, 
$x_{\alpha - 1}$ and $x_{\alpha - 2}$ are each greater than $1$.   By Proposition~\ref{pump}, it suffices to show that we can $1$-realize $(2,2)$, $(2,2,2)$, and any list of the form $(2,2,2,1,\ldots,1)$.  
Since all of these lists are initial segments or standard extensions of $(2,2,2,1)$, which is realized by  the \tp{1} graph $\graph{H}_7$ (see Figure~\ref{h7}), they are all $1$-realizable.
\end{proof}


\begin{lemma}\label{uoddp}  \
\begin{itemize}
\item[(i)] The list $(1,2, \ldots, 2)$ of length $\alpha$ is uniquely realized by $\graph{P}_{2\alpha -1}$.
\item[(ii)] The list $(1,2, \ldots,  2, 1)$ is not realizable.
\item[(iii)] The list $(2, \ldots, 2)$ of length $\alpha$ is uniquely $0$-realized by $\graph{P}_{2\alpha}$.
\item[(iv)] The list $(2, \ldots, 2, 1)$ is not $0$-realizable.
\end{itemize}
\end{lemma}

\begin{proof} \
{\bf Proof of (i)}: The statement is true by inspection for $\alpha =2$.  
It is clear that $\graph{P}_{2\alpha -1}$ realizes $(1,2, \ldots, 2)$.
We proceed by induction, with base case $\alpha =3$, to show the uniqueness.  

Base case ($\alpha=3$): Consider
any graph $\graph{G}$ realizing $(1,2,2)$; suppose $\X_3 = \{a\}$, $\X_2 = \{b_1, b_2\}$, and $\X_1 = \{c_1, c_2\}$.  The list $(1,2)$ is uniquely realized by $\graph{P}_3$, so $b_1$ and $b_2$ are not adjacent.  If they are both adjacent to $c_1$, then by \upcorner \ $a$ must strictly corner $c_1$.  In order for $b_1$ and $b_2$ to not be strictly cornered by $a$ in $\graph{G}$, they must each be adjacent to $c_2$ and $a$ must not. But then no vertex of rank 2 or 3 strictly corners $c_2$, contradicting \upcorner. Thus  each vertex of rank 2 has a unique neighbor of rank 1, so we assume that $b_1c_1, b_2c_2 \in E(\graph{G})$, while $b_1c_2, b_2c_1 \notin E(\graph{G})$. By Proposition~\ref{nbrrank}, $a$ cannot be adjacent to either $c_1$ or $c_2$, and thus for $i=1,2$, by Lemma~\ref{higran}, $c_i$ must be strictly cornered by $b_i$.  Thus $c_1c_2 \notin E(\graph{G})$, and $\graph{G} = \graph{P}_5$.

Inductive step: Now consider a graph \graph{G} with rank $\alpha \ge 4$ realizing the list $(1,2,\ldots ,2)$.  By the inductive hypothesis, $\uprank{\graph{G}}{2} = \graph{P}_{2\alpha-3} = (v_1, v_2, \ldots, v_{2\alpha-3})$. Since $\alpha \ge 4$, the shortest path in \uprank{\graph{G}}{2} between  $v_1$ and $v_{2\alpha-3}$ (which are the two rank 2 vertices in \graph{G}) has length at least four.  Let $y$ and $z$ be the two rank 1 vertices in \graph{G}  (see Figure~\ref{pp122}). By Proposition~\ref{nbrrank}, $v_1$ and $v_{2\alpha-3}$ must each be adjacent to some rank 1 vertex.  However, by \pathcon, $v_1$ and $v_{2\alpha-3}$ cannot both be adjacent to the same rank 1 vertex in \graph{G}, and furthermore, $y$ and $z$ cannot be adjacent, or else there is a path of length 2 or 3 between $v_1$ and $v_{2\alpha-3}$ in \graph{G}.  Thus without loss of generality, assume $yv_1, zv_{2\alpha - 3} \in E(\graph{G})$ and
$zv_1, yv_{2\alpha - 3} \not\in E(\graph{G})$.  To show that $\graph{G} = \graph{P}_{2 \alpha - 1}$ we just need to rule out edges of the form $yv_i$, where $v_i$ has rank at least 3 (an analogous discussion holds for $z$).  Suppose there is an edge $yv_i \in E(\graph{G})$ where $v_i$ has rank at least 3.  Then the vertex that strictly corners $y$ in \graph{G} is not $v_1$, but must be adjacent to $v_1$, and so must be $v_2$.  But in this case $v_2$ strictly corners $v_1$ in \graph{G}, contradicting the assumption that $v_1$ has rank 2.  So no edges from higher rank vertices to $y$ or $z$ are possible, and $\graph{G}= \graph{P}_{2\alpha-1}$.

{\bf Proof of (ii)}: Corollary~\ref{norca-1} and Proposition~\ref{norca-2} imply that $(1,1)$ and $(1,2,1)$ are not realizable.  For $\alpha \ge 4$, if \graph{G} is a graph realizing $(x_\alpha, \ldots, x_1)$ with $x_\alpha = x_1 = 1$  and $x_k = 2$ for $2 \le k < \alpha$, then by (i), $\uprank{\graph{G}}{2} = \graph{P}_{2\alpha-3}$ and the two rank 2 vertices 
$u$ and $v$ in \graph{G} have distance $2\alpha-4 \ge 4$ in \uprank{\graph{G}}{2}.
If there were one vertex of rank 1, then by Corollary~\ref{xk1} the rank 1 
vertex is adjacent to both $u$ and $v$, yielding a length 2 path from $u$ to $v$, contradicting \pathcon.

{\bf Proof of (iii)}: This proof is almost the same as the proof of (i), but now with a base case stating that $(2,2,2)$ is uniquely $0$-realized by $\graph{P}_6$; the proof of the base case is a similar technical proof to that of the base case for $(1,2,2)$. 

{\bf Proof of (iv)}: This proof is the same as the proof of (ii), using (iii) instead of (i).

\end{proof}

\begin{figure}
\begin{center}
    \begin{tikzpicture}[every node/.style={circle,fill=black}, scale=.8]  
      \node (v1) at (6,1) [label=above:$v_1$]{};
      \node (v2) at (4,1) [label=above:$v_2$] {};
      \node (v3) at (2,1) {};
      \node (v4) at (-2,.5) {};
      \node (v8) at (0,1) {};
      \node (v9) at (0,0) {};
     \node (v5) at (2,0) {};
      \node (v6) at (4,0) [label=below:$v_{2\alpha-4}$]{};
      \node (v7) at (6,0) [label=below:$v_{2\alpha-3}$]{};
      \node (y) at (8,1) [label=above:$y$]{};
       \node (z) at (8,0)  [label=below:$z$]{};
    \foreach \from/\to in {v1/v2, v9/v5, v6/v7, v3/v8, v8/v4, v4/v9}
    \draw[ultra thick] (\from) -- (\to);
    \draw[dotted] (v2) -- (v3);
    \draw[dotted] (v5) -- (v6);
    \end{tikzpicture}
\end{center}
\caption{The unique graph realizing $(1,2, \ldots, 2)$ is $\graph{P}_{2\alpha-1}$.}\label{pp122}
\end{figure}
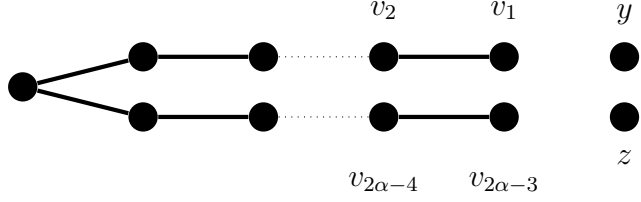

We now turn our attention to graphs with rank 4.
\begin{lemma} \label{lem_rank3or4}
If a graph realizes $(a,b,c,1)$ then there is a vertex of rank $3$ or $4$ that dominates the rank $2$ vertices.
\end{lemma}
\begin{proof}
Let \graph{G} be the graph and
let $d$ be the lone vertex in $\X_1$.  By Corollary~\ref{xk1}, $\X_2 \subseteq \clnb{d}$.  By \upcorner, \ some vertex $x$ of rank greater than 1 must strictly corner $d$, so $\X_2 \subseteq \clnb{x}$.  If $x \in \X_2$, then by \upcorner \ let $y$ be a vertex of rank at least 3 that strictly corners $x$ in \uprank{\graph{G}}{2},
otherwise let $y=x$. In either case, we have a vertex $y$ in either $\X_3$ or $\X_4$ such that $\X_2 \subseteq \clnb{y}$.
\end{proof}

\begin{lemma}\label{1mk1}
Suppose a graph realizes $(1,m,k,1)$.  Then the subgraph induced by the rank $3$ vertices is connected.
\end{lemma}

\begin{proof}
Let \graph{G} be the graph and 
let \graph{H} be the subgraph induced by the rank 3 vertices.
Assume for the sake of contradiction that the claim is false.
Suppose $a$ is the rank 4 vertex, two components of \graph{H} have vertex sets $B_1$ and $B_2$, and for $i=1,2$, $b_i \in B_i$.  By Proposition~\ref{nbrrank}, there must be a rank 2 vertex $c_1$ adjacent to $b_1$ but not to $a$.  Since $b_1$ is only adjacent to rank 3 vertices in $B_1$, by \upcorner, $c_1$ must be strictly cornered in \uprank{\graph{G}}{2} by a vertex in $B_1$ and thus $c_1$ is only adjacent to rank 3 vertices in $B_1$. Similarly, there is a rank 2 vertex $c_2$ that is adjacent to $b_2$, but not to $a$;
likewise, $c_2$ is only adjacent to rank 3 vertices in $B_2$.  If $c_1$ and $c_2$ are adjacent or have a common neighbor $c$ of rank 2 then the vertex of higher rank (which we have by \upcorner) that strictly corners $c$ (or $c_1$ if $c_1$ and $c_2$ are adjacent) in \uprank{\graph{G}}{2} would have to be adjacent to both $c_1$ and $c_2$. 
However, no such higher rank vertex exists since it would have to be in both $B_1$ and $B_2$, but these sets are disjoint.
Thus $c_1$ and $c_2$ are at distance at least three in \uprank{\graph{G}}{2}, and by \pathcon, they cannot both be adjacent to the single rank 1 vertex, contradicting Corollary~\ref{xk1}.
\end{proof}

\noindent
Since the graph induced by the rank 1 vertices of any graph realizing $(1,3)$ or $(1,2)$ is not connected, Lemma~\ref{1mk1} implies the following corollary.

\begin{cor} \label{cor_not12k1_not13k1}
For all $k \ge 1$, the lists $(1,2,k,1)$ and
$(1,3,k,1)$ are not realizable.

\end{cor}

\begin{lemma}\label{2mk1con} \
\begin{itemize}
\item[(i)] For $k \ge 1$, the list $(2,4,k,1)$ is not $0$-realizable.
\item[(ii)] The list $(2,5,2,1)$ is not $0$-realizable.
\end{itemize}
\end{lemma}

\begin{proof}
The proofs of (i) and (ii) are similar, with only some differences at the end. Consider for the sake of contradiction a graph \graph{G} that 0-realizes $(2,4,k,1)$ or $(2,5,2,1)$. 
Since \graph{G} is a \tp{0} graph and $\X_4 = \{a_1, a_2\}$  has only two vertices, there are rank 3 vertices $b_1$ and $b_2$ such that $a_1b_1, a_2b_2 \in E(\graph{G})$ and $a_1b_2, a_2b_1 \not\in E(\graph{G})$.  For $i=1,2$,  $a_i$ must strictly corner $b_i$ and every rank 3 neighbor of $b_i$ in \uprank{\graph{G}}{3}; we will use this point throughout the proof.
Since no rank 4 vertex is adjacent to both $b_1$ and $b_2$, they can  share no common neighbors in \uprank{\graph{G}}{3}
(since no rank 4 vertex could corner such a vertex in \uprank{\graph{G}}{3}), and by \pathcon, $b_1$ and $b_2$ must be at distance at least 3 in \graph{G}.
For $i=1,2$, let $c_i$ be a rank 2 vertex adjacent to $b_i$ but not $a_i$, which must exist by Proposition~\ref{nbrrank}. Since the distance between $b_1$ and $b_2$ is at least 3, $c_1$ and $c_2$ must be distinct vertices, and $b_1c_2, b_2c_1 \not\in E(\graph{G})$.  
 
Since no vertex of rank 4 dominates $\X_2$, by Lemma~\ref{lem_rank3or4}, there is a vertex $b_3$ of rank 3 that dominates $\X_2$, and $b_3$ is not $b_1$ or $b_2$.  Without loss of generality suppose $a_2$ corners $b_3$ in 
\uprank{\graph{G}}{3}.  
Now consider what corners $c_1$ in \uprank{\graph{G}}{2}: neither $a_i$, not $b_2$ because it is not adjacent to $c_1$, and neither $b_1$ nor $b_3$ since that would force $b_1$ and $b_3$ to be neighbors and would imply $a_2$ is adjacent to $b_1$, a contradiction.
So a fourth distinct rank 3 vertex $b_4$ must corner $c_1$ in \uprank{\graph{G}}{2}, and thus $b_4$ must be adjacent to both $b_1$ and $b_3$.

\emph{To finish the proof for (i):} 
Now consider what vertex of rank at least 3 strictly corners $c_2$ in \uprank{\graph{G}}{2}.  Since the distance from $b_1$ to $b_2$ is at least 3, neither of $b_1$ or $b_4$ can be adjacent to $b_2$ and thus neither of these vertices can corner $c_2$.  Neither vertex of rank 4 works since $a_1$ is not adjacent to $b_2$ and $a_2$ is not adjacent to $c_2$.  So $b_2$ or $b_3$ strictly corners $c_2$ in \uprank{\graph{G}}{2}, and are thus adjacent to each other. But now $b_2$ is strictly cornered by $b_3$ in \uprank{\graph{G}}{2},
since they have the same neighbors in \uprank{\graph{G}}{2}, except that $b_3$ is adjacent to $c_1$ and $b_4$, while $b_2$ is not.

\emph{To finish the proof for (ii):} 
Since $a_2$ is not adjacent to $b_1$, $a_1$ must corner $b_4$ in \uprank{\graph{G}}{3}, so in particular $a_1$ and $b_4$ are adjacent.
Since $b_1$ is not strictly cornered by $b_4$ in \uprank{\graph{G}}{2},  it must be adjacent to the fifth rank 3 vertex $b_5$, while $b_4$ and $b_5$ are not adjacent. 
Since $b_4$ corners $c_1$, $b_5$ is not adjacent to $c_1$, so by Proposition~\ref{nbrrank}, $b_5$ must be adjacent to $c_2$.
Since $a_1$ must strictly corner $b_5$ in \uprank{\graph{G}}{3}, $b_5$ is not adjacent to $b_2$, and thus $b_3$ is the only vertex that can strictly corner $c_2$ in \uprank{\graph{G}}{2}.  But then $b_3$ is adjacent to the rank 3 vertices $b_2$, $b_4$, and $b_5$, and thus also $a_1$.  Thus in \uprank{\graph{G}}{3}, $b_3$ has at least the neighbors that $a_2$ has, contradicting the fact that $a_2$ strictly corners $b_3$ in \uprank{\graph{G}}{3}.
\end{proof}

\begin{lemma}\label{m2k1}
For any $m, k \ge 1$, $(m,2,k,1)$ is not $0$-realizable.
\end{lemma}

\begin{proof}
For the sake of contradiction, suppose \graph{G} $0$-realizes $(m,2,k,1)$. Let $\X_3 = \{b_1, b_2\}$ and note that every rank 4 vertex is adjacent to exactly one of these two vertices.  Thus $b_1b_2 \not\in E(\graph{G})$, and these two vertices are at distance 3 in \uprank{\graph{G}}{3} and hence in \graph{G}. Thus by \pathcon \ they share no rank two neighbors. 
By Lemma~\ref{lem_rank3or4}, there is a vertex $x$ 
of rank 3 or 4 that dominates $\X_2$. 
 Since $b_1$ and $b_2$ must both have rank 2 neighbors but can't have any in common, neither of these vertices can be $x$.  Thus $x$ must be a rank 4 vertex.  But if $x$ is adjacent to $b_i$, then it strictly corners $b_i$ in \uprank{\graph{G}}{2}, contradicting the assumption that $b_i$ has rank 3.
\end{proof}

\junk{
\begin{defn}
Define $\graph{H}_7$ to  be the graph with vertex set $\{a_1, a_2, b_1, b_2, c_1, c_2, d\}$ and edge set $\{a_1a_2, a_1b_1, a_1b_2, a_2b_2, a_2b_1, a_2c_1, a_1c_2, b_1c_1, b_1c_2, b_1d, b_2c_1, b_2c_2, c_1d, c_2d\}$. 
\end{defn}
}

\begin{lemma}\label{H7only}
The list (2,2,2,1) is uniquely realized by the graph $\graph{H}_7$.
\end{lemma}

\begin{proof}
Recall that the graph $\graph{H}_7$ is displayed in Figure~\ref{h7}.
Let \graph{G} be a graph that realizes $(2,2,2,1)$, with $\X_4 = \{a_1, a_2\}$, $\X_3 = \{b_1, b_2\}$, $\X_2 = \{c_1, c_2\}$, and $\X_1 = \{d\}$.  Lemma~\ref{uoddp} implies that $(2,2,2,1)$ is not $0$-realizable. Thus  \uprank{\graph{G}}{3} must contain the edges $a_1a_2$, $a_1b_1$, $a_2b_2$, $a_1b_2$, $a_2b_1$, and since \uprank{\graph{G}}{3} is not a clique, there is not an edge $b_1b_2$. By Corollary~\ref{xk1}, each of $b_1$ and $b_2$ must be adjacent to a vertex of rank 2, and Lemma~\ref{lem_rank3or4} implies that some vertex $x$ of rank 3 or 4 dominates $\X_2$.  If $x$ were some $a_i$, then $x$ would strictly corner each rank 3 vertex in \uprank{\graph{G}}{2}, a contradiction. Thus without loss of generality we may assume $b_1$ dominates $\X_2$ and both $b_1$ and $b_2$ are adjacent to $c_1$. Then only a vertex from $\X_4$ can strictly corner $c_1$ in \uprank{\graph{G}}{2}; without loss of generality, suppose $a_2$ is this vertex, so in particular, $a_2$ is adjacent to $c_1$. Since $a_2$ is not a dominating vertex in \uprank{\graph{G}}{2}, it cannot be adjacent to $c_2$ and thus $c_1$ and $c_2$ cannot be adjacent. For $a_2$ not to strictly corner $b_2$ or $a_1$ in \uprank{\graph{G}}{2}, each of these vertices must be adjacent to $c_2$, and $a_1$ cannot be adjacent to $c_1$ or else it dominates \uprank{\graph{G}}{2}. Thus $\uprank{\graph{G}}{2}= \uprank{(\graph{H}_7)}{2}$.

By Corollary~\ref{xk1}, the rank 1 vertex $d$ is adjacent to both $c_1$ and $c_2$, which means it can only be strictly cornered by some $b_i$, without loss of generality, $b_1$.  Since the rank 3 vertices are not adjacent, $d$ cannot be adjacent to $b_2$.  
To see that $d$ is not adjacent to any rank 4 vertices, 
first note that from the above discussion, we can conclude that in \uprank{\graph{G}}{2},
$a_2$ strictly corners $c_1$ and $a_1$ strictly corners $c_2$.
Then by Proposition~\ref{nbrrank} (applied with $w = a_2, v = c_1$, and then $w = a_1, v = c_2$),  $d$ cannot be adjacent to either $a_1$ or $a_2$.
Thus \graph{G} is $\graph{H}_7$.
\end{proof}

\section{A Characterization of $\mathcal{G}^{n-4}_n$, the CT-Maximal Graphs}\label{nvertsec}

We can now characterize the rank cardinality lists of all the CT-maximal graphs. The following definition will be used to classify the CT-maximal graphs having at least seven vertices.
\begin{defn} \label{def_h7+}
For $k \ge 0$, define $\mathcal{H}_7^{+k}$ to be a set of graphs 
that realize the length $k+4$ list of the form $(2,2,2,1, \ldots, 1)$.
Let $\mathcal{H}_7^+$ be  
$\bigcup_{k \ge 0} \mathcal{H}_7^{+k}$.
\end{defn}
\noindent
For example, Lemma~\ref{H7only} implies that $\mathcal{H}_7^{+0} = \{\graph{H}_7\}$.  Figure~\ref{h7+} displays some of the graphs in $\mathcal{H}_7^{+1}$. By Proposition~\ref{extend}, any standard extension of $(2,2,2,1)$ is realizable, so for each $k$, $\mathcal{H}_7^{+k}$ is non-empty. In \cite{Gav2010}, $\mathcal{M}$ is defined to be the set of CT-maximal graphs.  We will see (in Theorem~\ref{gave})  that for graphs with order at least 9, the graphs of $\mathcal{H}_7^+$  are exactly the graphs  in
$\mathcal{M}$.  In Theorem 2 of \cite{Gav2010} a nice, but somewhat involved characterization of $\mathcal{M}$ is given (stated to be true for $n \ge 8$, but actually true for $n \ge 9$).  Our result gives a simpler characterization (for $n \ge 9$): A graph is in $\mathcal{M}$ exactly when it realizes $(2,2,2,1, \ldots, 1)$.
In Theorem 2 of \cite{Gav2010}, Gavenciak demonstrates that various properties hold for
the graphs in $\mathcal{M}$ of order 9 and larger. 
Our approach also demonstrates that these properties hold.  The properties follow immediately from our characterization of   
$\mathcal{M}$ as equaling $\mathcal{H}_7^+$ (in Theorem~\ref{gave}), together with the next
theorem.

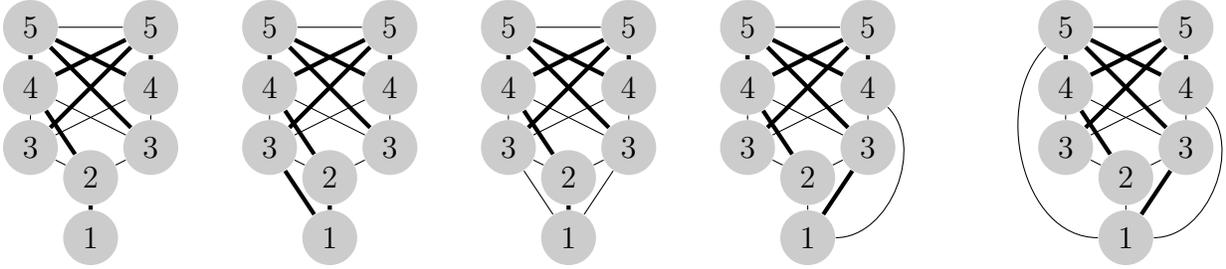
\begin{figure}

\begin{center}
  \begin{tikzpicture}[every node/.style={circle,fill=black!20}, scale=.8] 
      \node (a1) at (0,5)  {$5$};
      \node (a2) at (2,5)  {$5$};
      \node (b1) at (0,4)  {$4$};
      \node (b2) at (2,4)  {$4$};
      \node (c1) at (0,3)  {$3$};
      \node (c2) at (2,3)   {$3$};
      \node (d) at (1,2.5)   {$2$};
     \node (e) at (1,1.5)   {$1$};
    \foreach \from/\to in {a1/a2,  b1/c1, b1/c2, b2/c1, b2/c2, c1/d, c2/d}
    \draw (\from) -- (\to);
    \foreach \from/\to in {a1/b2, a1/b1, a2/b2, a2/b1, a2/c1, a1/c2, b1/d, d/e}
    \draw[ultra thick] (\from) -- (\to);
    \end{tikzpicture}
\hfill
  \begin{tikzpicture}[every node/.style={circle,fill=black!20}, scale=.8] 
      \node (a1) at (0,5)  {$5$};
      \node (a2) at (2,5)  {$5$};
      \node (b1) at (0,4)  {$4$};
      \node (b2) at (2,4)  {$4$};
      \node (c1) at (0,3)  {$3$};
      \node (c2) at (2,3)   {$3$};
      \node (d) at (1,2.5)   {$2$};
     \node (e) at (1,1.5)  {$1$};
    \foreach \from/\to in {a1/a2,  b1/c1, b1/c2, b2/c1, b2/c2, c1/d, c2/d}
    \draw (\from) -- (\to);
    \foreach \from/\to in {a1/b2, a1/b1, a2/b2, a2/b1, a2/c1, a1/c2, b1/d, c1/e, d/e}
    \draw[ultra thick] (\from) -- (\to);
    \end{tikzpicture}
\hfill
  \begin{tikzpicture}[every node/.style={circle,fill=black!20}, scale=.8] 
      \node (a1) at (0,5)  {$5$};
      \node (a2) at (2,5)  {$5$};
      \node (b1) at (0,4)  {$4$};
      \node (b2) at (2,4)  {$4$};
      \node (c1) at (0,3)  {$3$};
      \node (c2) at (2,3)   {$3$};
      \node (d) at (1,2.5)   {$2$};
     \node (e) at (1,1.5)   {$1$};
    \foreach \from/\to in {a1/a2,  b1/c1, b1/c2, b2/c1, b2/c2, c1/d, c2/d, c1/e, c2/e}
    \draw (\from) -- (\to);
    \foreach \from/\to in {a1/b2, a1/b1, a2/b2, a2/b1, a2/c1, a1/c2, b1/d, d/e}
    \draw[ultra thick] (\from) -- (\to);
    \end{tikzpicture}
    \hfill
  \begin{tikzpicture}[every node/.style={circle,fill=black!20}, scale=.8] 
      \node (a1) at (0,5)  {$5$};
      \node (a2) at (2,5)  {$5$};
      \node (b1) at (0,4)  {$4$};
      \node (b2) at (2,4)  {$4$};
      \node (c1) at (0,3)  {$3$};
      \node (c2) at (2,3)   {$3$};
      \node (d) at (1,2.5)   {$2$};
     \node (e) at (1,1.5)   {$1$};
    \foreach \from/\to in {a1/a2,  b1/c1, b1/c2, b2/c1, b2/c2, c1/d, c2/d, d/e}
    \draw (\from) -- (\to);
    \foreach \from/\to in {a1/b2, a1/b1, a2/b2, a2/b1, a2/c1, a1/c2, b1/d, c2/e}
    \draw[ultra thick] (\from) -- (\to);
        \draw  (e) to [out=0,in=315] (b2);
    \end{tikzpicture}
        \hfill
  \begin{tikzpicture}[every node/.style={circle,fill=black!20}, scale=.8] 
      \node (a1) at (0,5)  {$5$};
      \node (a2) at (2,5)  {$5$};
      \node (b1) at (0,4)  {$4$};
      \node (b2) at (2,4)  {$4$};
      \node (c1) at (0,3)  {$3$};
      \node (c2) at (2,3)   {$3$};
      \node (d) at (1,2.5)   {$2$};
     \node (e) at (1,1.5)   {$1$};
    \foreach \from/\to in {a1/a2,  b1/c1, b1/c2, b2/c1, b2/c2, c1/d, c2/d, d/e}
    \draw (\from) -- (\to);
    \foreach \from/\to in {a1/b2, a1/b1, a2/b2, a2/b1, a2/c1, a1/c2, b1/d, c2/e}
    \draw[ultra thick] (\from) -- (\to);
      \draw  (e) to [out=0,in=315] (b2);
      \draw  (e) to [out=180,in=225] (a1);
    \end{tikzpicture}
\end{center}
\caption{Some graphs in $\mathcal{H}_7^{+1}$.}\label{h7+}
\end{figure}

\begin{theorem} \label{thm_H7}
Suppose \graph{G} is a graph on $n$ vertices in $\mathcal{H}_7^+$.  Then

\begin{itemize}

\item[(i)] \uprank{\graph{G}}{\alpha - 3}  is $\graph{H}_7$.


\item[(ii)] $\capt(\graph{G}) = n - 4$.

\end{itemize}

\end{theorem}

\begin{proof}
Property (i) follows from Lemma~\ref{H7only}. 
For Property (ii), note that \graph{G} has rank $n - 3$ and is \tp{1}.  
Thus by Theorem~\ref{thm_rank_capt_time}, \graph{G} has capture time $(n - 3) - 1 = n-4$.
\end{proof}

The next theorem restates the main results of \cite{Gav2010}, with an alternative proof that does not use a computer search.


\begin{theorem} \label{gave}
For $n \ge 7$, $\capt(n) = n - 4$, and for graphs on at least $9$ vertices, the CT-maximal graphs are exactly the graphs in
$\mathcal{H}_7^{+}$.  Furthermore, in Table~\ref{table_thm}, we describe $\capt(n)$ and the CT-maximal graphs for $n \le 8$.

\begin{table}[h!]
  \centering

\begin{tabular}{l|l|l}
$n$  & $\capt(n)$ & CT-Maximal Graphs with $n$ vertices\\ \hline
1 & 0 & $\graph{P}_1$\\ 
2 & 1 & $\graph{P}_2$\\
3 & 1 & $\graph{P}_3$, $\graph{K}_3$ \\
4 & 2 & $\graph{P}_4$\\
5 &  2 & $\graph{P}_5$ and the \tp{0} graphs realizing (2,3) and (3,2)\\
6 & 3 & $\graph{P}_6$ \\
7 & 3 & $\graph{P}_7$, $\graph{H}_7$, and the \tp{0} graphs realizing 
(2,2,3), (2,3,2), (3,2,2) \\
8 & 4 & $\graph{P}_8$ and any graph in $\mathcal{H}_7^{+1}$ 
\end{tabular}  
\caption{CT-Maximal graphs with at most 8 vertices and their capture time} \label{table_thm}
\end{table}

\end{theorem}

\vspace{3mm}

\begin{proof}
The bulk of the proof will focus on the part of the theorem which classifies the structure of CT-Maximal graphs.  Once this structural result is demonstrated, we can
quickly conclude that 
$\capt(n) = n - 4$ for $n \ge 7$.
By Theorem~\ref{thm_H7}, any graph in $\mathcal{H}_7^+$ with $n$ vertices has capture time $n - 4$, as required. The 
other relevant graphs (the ones of order 7 and 8
listed in Table~\ref{table_thm}) all have the required capture time, since the paths $\graph{P}_7$ and $\graph{P}_8$, and the 
\tp{0} graphs of order 7 in Table~\ref{table_thm} all have capture time 3 by Theorem~\ref{thm_rank_capt_time}.

Now we prove the structural classification, first for graphs
where $n \le 8$ and then for graphs where $n \ge 9$.
For the case of $n \le 8$, consider lists realized by $\graph{P}_n$. 
By Lemma~\ref{uoddp}, when  $n$ is even, $\graph{P}_n$ is the unique \tp{0} graph realizing the length $n/2$ list $(2, \ldots, 2)$, and when  $n$ is odd, $\graph{P}_n$ is the unique \tp{1} graph realizing the length $\lceil n/2 \rceil$ list  $(1,2, \ldots, 2)$.  
Thus when $n$ is even, graphs whose rank cardinality list has length less than $n/2$ cannot be CT-maximal, and when $n$ is odd, graphs whose rank cardinality list has length less than $\lfloor n/2 \rfloor$ cannot be CT-maximal.
Based on this observation, Table~\ref{table_proof} lists all lists with sum $n \le 8$ that could possibly be the rank cardinality list of some CT-maximal graph;
by Corollary~\ref{norca-1} and Proposition~\ref{norca-2}, we exclude the lists whose second or third entry is 1.
Note that the first list (in {\bf bold}) is the rank cardinality list for the corresponding path
$\graph{P}_n$.

\begin{table}[h!]
  \centering

\begin{tabular}{l|l}
$n$  &Vectors \\ \hline
1 & {\bf (1)} \\ 
2 & {\bf (2)} \\
3 &  {\bf (1,2)}, (3) \\
4 & {\bf (2,2)}, (1,3) \\
5 & {\bf (1,2,2)}, (1,4), (3,2), (2,3) \\
6 & {\bf (2,2,2)}, (1,3,2), (1,2,3), (1,2,2,1) \\ 
7 & {\bf (1,2,2,2)}, (2,2,2,1), (2,2,3), (2,3,2), (3,2,2), (1,2,2,1,1), (1,3,2,1), (1,2,3,1) \\
8 & {\bf (2,2,2,2)}, (2,2,2,1,1)
\end{tabular}

  \caption{Vectors with sum $n\le 8$ and length at least $\lfloor n/2 \rfloor$.}\label{table_proof}

\end{table}

To prove the theorem for  $n \le 8$, it suffices to show that each list is either: 1) not realizable, 2) has capture time less than that of $\graph{P}_n$, or 3) is accounted for in Table~\ref{table_thm}.
We proceed by cases on the values of $n \le 8$, employing Theorem~\ref{thm_rank_capt_time} and using the immediate fact that if the first entry is 1,
then a graph that realizes the list must be \tp{1}.
 At various points in this proof all we need to show is that some list is realizable; in some of those cases, as an interesting tangent, we claim that the list is uniquely realized, or we produce all the graphs realizing the list.

\begin{itemize}

\item
For $n=1,2,3$, all the lists in Table~\ref{table_proof} have corresponding graphs listed in Table~\ref{table_thm}.

\item
For $n = 4$, a graph realizing $(1,3)$ has capture time $1 < 2$, so it is not CT-maximal.

\item
For $n=5$, besides $(1,2,2)$,
the lists in Table~\ref{table_proof} have length less than 3, so they can only have capture time 2 if they are \tp{0}. Thus we also get as CT-maximal graphs the unique graph $0$-realizing $(3,2)$ and the three graphs $0$-realizing $(2,3)$. (See Figure~\ref{2332}.)

\begin{figure}
\begin{center}
\begin{tikzpicture}[every node/.style={circle,fill=black!20}, scale=.65]  
      \node (a1) at (2.5,2){$2$};
      \node (a2) at (5.5,2){$2$};
      \node (a3) at (4,2.5){$2$};
      \node (b1) at (3.25,1) {$1$};
      \node (b2) at (4.75,1) {$1$};
    \foreach \from/\to in {a1/a2,a1/a3, a2/a3, a1/b1, a2/b2, a3/b1}
    \draw (\from) -- (\to);
    \foreach \from/\to in {a1/b1, a3/b1, a2/b2}
 \draw[ultra thick] (\from) -- (\to);
    \end{tikzpicture}
   \begin{tikzpicture}[every node/.style={circle,fill=black!20}, scale=.65]  
      \node (a1) at (1,2){$2$};
      \node (a2) at (3,2){$2$};
      \node (b1) at (0,1) {$1$};
      \node (b2) at (2,1) {$1$};
      \node (b3) at (4,1) {$1$};
    \foreach \from/\to in {a1/a2,a1/b1, a1/b2, a2/b3}
    \draw (\from) -- (\to);
        \foreach \from/\to in {a1/b1, a1/b2, a2/b3}
    \draw[ultra thick] (\from) -- (\to);
    \end{tikzpicture}
    \begin{tikzpicture}[every node/.style={circle,fill=black!20}, scale=.65]  
      \node (a1) at (1,2){$2$};
      \node (a2) at (3,2){$2$};
      \node (b1) at (0,1) {$1$};
      \node (b2) at (2,1) {$1$};
      \node (b3) at (4,1) {$1$};
    \foreach \from/\to in {a1/a2,a1/b1, a1/b2, a2/b3, b1/b2}
    \draw (\from) -- (\to);
            \foreach \from/\to in {a1/b1, a1/b2, a2/b3}
    \draw[ultra thick] (\from) -- (\to);
    \end{tikzpicture}
       \begin{tikzpicture}[every node/.style={circle,fill=black!20}, scale=.65]  
      \node (a1) at (1,2){$2$};
      \node (a2) at (3,2){$2$};
      \node (b1) at (0,1) {$1$};
      \node (b2) at (2,1) {$1$};
      \node (b3) at (4,1) {$1$};
    \foreach \from/\to in {a1/a2,a1/b1, a1/b2, a2/b3, a2/b2}
    \draw (\from) -- (\to);
            \foreach \from/\to in {a1/b1, a1/b2, a2/b2, a2/b3}
    \draw[ultra thick] (\from) -- (\to);
        \end{tikzpicture}
        \end{center}
\caption{The unique graph $0$-realizing $(3,2)$ and the three graphs $0$-realizing $(2,3)$.}\label{2332}
 \end{figure}
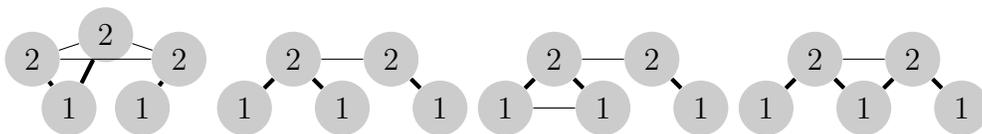

\item
For $n = 6$, the only list, besides $(2,2,2)$, corresponding to a capture time of  3 or greater is
$(1,2,2,1)$, but that list is not realizable, by Lemma~\ref{uoddp}.

\item
For $n = 7$, the list $(2,2,2,1)$ is uniquely realized by $\graph{H}_7$, using Lemma~\ref{H7only}.  To achieve the required capture time of 3, we can also take one of the five graphs $0$-realizing $(2,2,3)$ or one of the unique graphs $0$-realizing $(2,3,2),$ or $(3,2,2)$. (See Figure~\ref{223232}.)
The rest of the lists are not realizable:
$(1,2,2,1,1)$ is not realizable by Lemma~\ref{uoddp}, and
$(1,3,2,1)$ and $(1,2,3,1)$ are not realizable by Corollary~\ref{cor_not12k1_not13k1}.

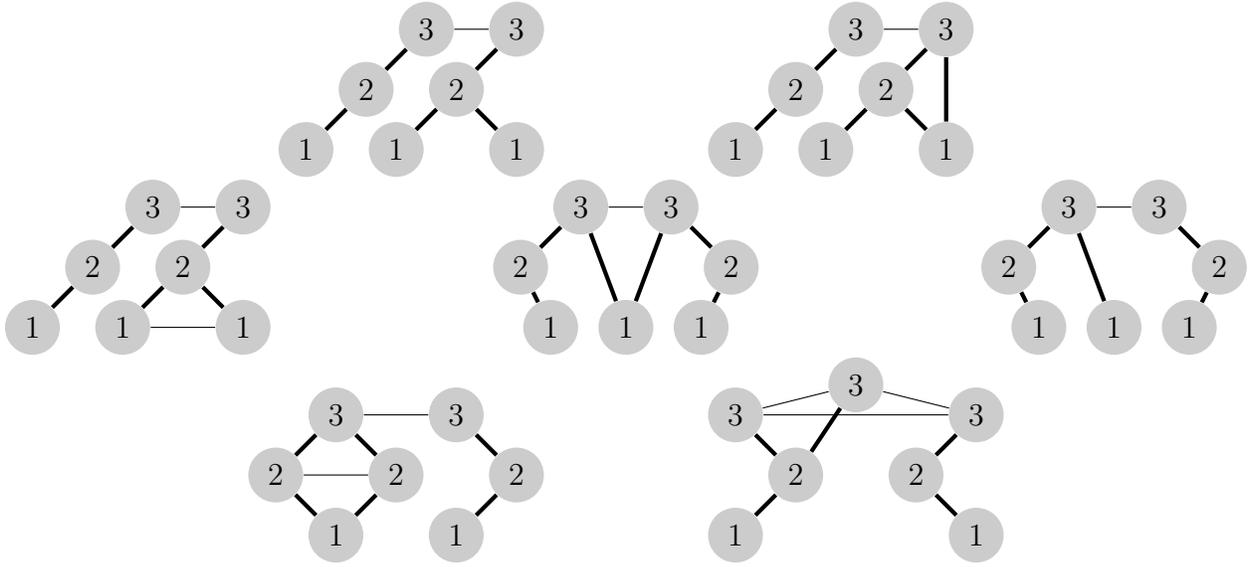
\begin{figure}
\begin{center}
 \begin{tikzpicture}[every node/.style={circle,fill=black!20}, scale=.8]  
      \node (a1) at (1,2){$3$};
      \node (a2) at (2.5,2){$3$};
      \node (b1) at (0,1) {$2$};
      \node (b2) at (1.5,1) {$2$};
      \node (c1) at (-1,0) {$1$};
      \node (c2) at (.5,0) {$1$};
      \node (c3) at (2.5,0) {$1$};
    \foreach \from/\to in {a1/a2}
    \draw (\from) -- (\to);
        \foreach \from/\to in {a1/b1, a2/b2, b1/c1, b2/c2,b2/c3}
    \draw[ultra thick] (\from) -- (\to);
    \end{tikzpicture}
    \hspace{.75in}
     \begin{tikzpicture}[every node/.style={circle,fill=black!20}, scale=.8]  
      \node (a1) at (1,2){$3$};
      \node (a2) at (2.5,2){$3$};
      \node (b1) at (0,1) {$2$};
      \node (b2) at (1.5,1) {$2$};
      \node (c1) at (-1,0) {$1$};
      \node (c2) at (.5,0) {$1$};
      \node (c3) at (2.5,0) {$1$};
    \foreach \from/\to in {a1/a2}
    \draw (\from) -- (\to);
        \foreach \from/\to in {a1/b1, a2/b2, b1/c1, b2/c2,b2/c3,a2/c3}
    \draw[ultra thick] (\from) -- (\to);
    \end{tikzpicture}\\
         \begin{tikzpicture}[every node/.style={circle,fill=black!20}, scale=.8]  
      \node (a1) at (1,2){$3$};
      \node (a2) at (2.5,2){$3$};
      \node (b1) at (0,1) {$2$};
      \node (b2) at (1.5,1) {$2$};
      \node (c1) at (-1,0) {$1$};
      \node (c2) at (.5,0) {$1$};
      \node (c3) at (2.5,0) {$1$};
    \foreach \from/\to in {a1/a2,c2/c3}
    \draw (\from) -- (\to);
        \foreach \from/\to in {a1/b1, a2/b2, b1/c1, b2/c2,b2/c3}
    \draw[ultra thick] (\from) -- (\to);
    \end{tikzpicture}
    \hfill
     \begin{tikzpicture}[every node/.style={circle,fill=black!20}, scale=.8]  
      \node (a1) at (1,2){$3$};
      \node (a2) at (2.5,2){$3$};
      \node (b1) at (0,1) {$2$};
      \node (b2) at (3.5,1) {$2$};
      \node (c1) at (.5,0) {$1$};
      \node (c2) at (1.75,0) {$1$};
      \node (c3) at (3,0) {$1$};
    \foreach \from/\to in {a1/a2}
    \draw (\from) -- (\to);
        \foreach \from/\to in {a1/b1, a2/b2, b1/c1, b2/c3,a1/c2, a2/c2}
    \draw[ultra thick] (\from) -- (\to);
    \end{tikzpicture}
    \hfill
    \begin{tikzpicture}[every node/.style={circle,fill=black!20}, scale=.8]  
      \node (a1) at (1,2){$3$};
      \node (a2) at (2.5,2){$3$};
      \node (b1) at (0,1) {$2$};
      \node (b2) at (3.5,1) {$2$};
      \node (c1) at (.5,0) {$1$};
      \node (c2) at (1.75,0) {$1$};
      \node (c3) at (3,0) {$1$};
    \foreach \from/\to in {a1/a2}
    \draw (\from) -- (\to);
        \foreach \from/\to in {a1/b1, a2/b2, b1/c1, b2/c3,a1/c2}
    \draw[ultra thick] (\from) -- (\to);
    \end{tikzpicture}\\
\begin{tikzpicture}[every node/.style={circle,fill=black!20}, scale=.8]  
      \node (a1) at (1,2){$3$};
      \node (a2) at (3,2){$3$};
      \node (b1) at (0,1) {$2$};
      \node (b2) at (2,1) {$2$};
      \node (b3) at (4,1) {$2$};
      \node (c1) at (1,0) {$1$};
      \node (c2) at (3,0) {$1$};
    \foreach \from/\to in {a1/a2, b1/b2}
    \draw (\from) -- (\to);
        \foreach \from/\to in {a1/b1, a1/b2, a2/b3, b1/c1, b2/c1, b3/c2}
    \draw[ultra thick] (\from) -- (\to);
    \end{tikzpicture}
    \hspace{.75in}
     \begin{tikzpicture}[every node/.style={circle,fill=black!20}, scale=.8]  
      \node (a1) at (2,2){$3$};
      \node (a2) at (6,2){$3$};
      \node (a3) at (4,2.5){$3$};
      \node (b1) at (3,1) {$2$};
      \node (b2) at (5,1) {$2$};
     \node (c1) at (2,0) {$1$};
      \node (c2) at (6,0) {$1$};
    \foreach \from/\to in {a1/a2,a1/a3, a2/a3}
    \draw (\from) -- (\to);
    \foreach \from/\to in {a1/b1, a2/b2, a3/b1, b1/c1, b2/c2}
    \draw[ultra thick] (\from) -- (\to);
    \end{tikzpicture}
    \end{center}
\caption{{\bf Top}: The five graphs $0$-realizing $(2,2,3)$.  
{\bf Bottom}: The unique graphs $0$-realizing $(2,3,2)$ and $(3,2,2)$.}\label{223232}
 \end{figure}

\item
For $n = 8$, by definition, the list $(2,2,2,1,1)$ is only realized by graphs from $\mathcal{H}_7^{+1}$.

\end{itemize}

Now we consider $n \ge 9$. We show that $\graph{H}_7^{+(n-7)}$ contains all the CT-maximal graphs.
For $\graph{H}_7^{+(n-7)}$ not to contain all the CT-maximal graphs we would need a realizable list 
$\vect{x} = (x_{\alpha}, \ldots, x_1)$ besides 
$(2,2,2, 1, 1, \ldots, 1)$ with one of the following properties.

\begin{itemize}

\item Type 0:
$\alpha \ge n -4$ and \vect{x} is $0$-realizable.

\item Type 1:
$\alpha \ge n -3$ and \vect{x} is $1$-realizable.

\end{itemize}
We show that no such lists are realizable.  Keep in mind that in both cases $x_{\alpha -1}$ and $x_{\alpha -2}$ must be at least 2 by Corollary~\ref{norca-1} and Proposition~\ref{norca-2}.

We rule out the Type 0 lists.
Let $\vect{y} = (y_{\alpha}, \ldots, y_1)$ be the list
$(2,2,2,1, \ldots, 1)$.  Being $0$-realizable, $x_{\alpha} \ge 2$.
Since $\alpha \ge n-4$, such an \vect{x} would be an augmentation of \vect{y} where all entries of \vect{x} are the same as the entries of \vect{y} with the possible exception of one entry of \vect{y}, which is one larger than its corresponding entry in \vect{x}.  No matter where the 1 is added, or if nothing is added, one of the following lists must be an initial segment of \vect{x}:
(3,2,2,1), (2,3,2,1), (2,2,3,1), (2,2,2,1) or (2,2,2,2,1).
The first and third lists are not $0$-realizable by Lemma~\ref{m2k1}, and the second is not $0$-realizable by Lemma~\ref{2mk1con};
the last two lists are not $0$-realizable by Lemma~\ref{uoddp}.

Now we rule out the Type 1 lists. 
Let $\vect{y} = (y_{\alpha}, \ldots, y_1)$ be the list
$(1,2,2,1, \ldots, 1)$.  Since $\alpha \ge n-3$, such an \vect{x} would be an augmentation of \vect{y} where all entries of \vect{x} are the same as the entries of \vect{y} with the possible exception of one entry of \vect{y}, which is one larger than its corresponding entry in \vect{x}.   The value 1 cannot be added to $y_{\alpha}$ since that would mean \graph{G} is in $\mathcal{H}_7^{+}$.  No matter where else 1 is added, or if nothing is added, one of the following lists must be an initial segment of \vect{x}:
(1,2,2,1), (1,2,2,2,1), (1,3,2,1), (1,2,3,1).
By Lemma~\ref{uoddp} the first two lists are not realizable, and by  Corollary~\ref{cor_not12k1_not13k1}, the last two lists are not realizable.

\end{proof}

\section{A Characterization of $\mathcal{G}_{n-5}^n$}\label{morrc}

Before we prove our second main result, Theorem~\ref{thm_n-5}, we need the following lemma.

\begin{lemma}\label{pathlike}
Let $x_\alpha, \ldots, x_1$ have the property that $x_j= 3$ for some $j>1$, and $x_i=2$ for all $i\neq j$. Then
\begin{itemize}
\item[(i)] There is exactly one graph that realizes $(1, x_\alpha, \ldots, x_1)$. 
\item[(ii)] $(1, x_\alpha, \ldots, x_1,1)$ is not realizable. 
\item[(iii)] There is exactly one graph that $0$-realizes $(x_\alpha, \ldots, x_1)$.
\item[(iv)] $(x_\alpha, \ldots, x_1,1)$ is not $0$-realizable.
\end{itemize}
\end{lemma}

\begin{proof} \

{\bf Proof of (i)}: 

We first suppose we have a list \vect{x} of the form $(1,3,2)$ or $(1,2, \ldots, 2, 3, 2)$, and we will show that it is uniquely realized, so we let \graph{G} be this unique graph.
If \vect{x} is $(1,3,2)$ or $(1,2,3,2)$, we will show that the corresponding graph \graph{G} is drawn in Figure~\ref{1232}.  Otherwise, we are considering an \vect{x} of length at least 5, of the form $(1,2, \ldots, 2, 3, 2)$; in this case the corresponding graph \graph{G} is partially drawn on the right side of Figure~\ref{1232}:
its bottom four ranks are drawn; also there are no edges between $V(\uprank{\graph{G}}{5})$ and any vertex of rank less than 4.   Once we have shown that such a list \vect{x} corresponds to such a unique graph \graph{G}, we can quickly obtain the uniqueness claim for any list which is a standard extension of \vect{x}.  Considering any such standard extension of \vect{x}, using the properties of \graph{G}, and key facts like \pathcon, we can see that any such standard extension is only realized by attaching an appropriate length path to each of the rank 1 vertices of \graph{G}.  The bulk of the proof now consists in showing that lists of the form \vect{x} are uniquely realized in the manner described.

We first deal with the cases of $(1,3,2)$ and $(1,2,3,2)$.  It is a simple exercise to see there is only one graph that realizes $(1,3,2)$ (see Figure ~\ref{1232}).
Now we show that there is only one graph that realizes $(1,2,3,2)$. Suppose \graph{G} realizes $(1,2,3,2)$, with $\X_4 = \{a\}$, $\X_3 = \{b_1, b_2\}$, $\X_2 = \{c_1, c_2, c_3\}$ and $\X_1 = \{d_1, d_2\}$.  There are 4 graphs realizing $(1,2,3)$ (note to the reader: in finding them, note that two have an edge between $a$ and $\X_1$, and two do not). 
In each of the 4 graphs we can assume without loss of generality that $b_1$ is adjacent only to $a$ and $c_1$, and $c_1$ has degree 1. Thus $c_1$ is at distance at least 3 from any other rank 2 vertex of \graph{G}, and
in any realization of $(1,2,3,2)$, $c_1$ must be adjacent to a vertex $d_1$ that is not adjacent to $b_1$, $c_2$ or $c_3$. This implies $c_1$ must strictly corner $d_1$.  The vertex $d_2$ must be adjacent to $c_2$ and $c_3$, and the only way to fill in the rest of the edges leads to Figure~\ref{1232} (to help see this, note that neither $b_2$ nor $a$ can strictly corner $d_2$).

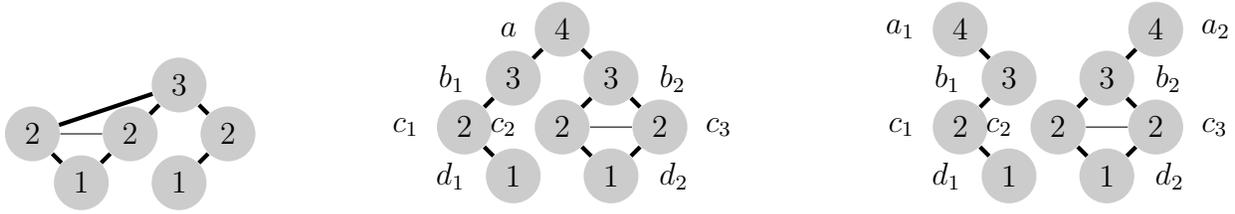
\begin{figure}
\begin{center}
\begin{tikzpicture}[every node/.style={circle,fill=black!20}, scale=.65]  
      \node (a1) at (1,2){$3$};
      \node (b1) at (-2,1) {$2$};
      \node (b2) at (0,1) {$2$};
      \node (b3) at (2,1) {$2$};
       \node (c1) at (-1,0) {$1$};
      \node (c2) at (1,0) {$1$};
    \foreach \from/\to in {b1/b2}
    \draw (\from) -- (\to);
        \foreach \from/\to in {a1/b1, a1/b2, a1/b3, b1/c1, b2/c1, b3/c2}
    \draw[ultra thick] (\from) -- (\to);
    \end{tikzpicture}
\hfill
    \begin{tikzpicture}[every node/.style={circle,fill=black!20}, scale=.65]  
      \node (a1) at (1,2)[label=left:$a$]{$4$};
      \node (b1) at (0,1) [label=left:$b_1$]{$3$};
      \node (b2) at (2,1) [label=right:$b_2$]{$3$};
      \node (c1) at (-1,0) [label=left:$c_1$]{$2$};
      \node (c2) at (1,0) [label=left:$c_2$]{$2$};
      \node (c3) at (3,0) [label=right:$c_3$]{$2$};
      \node (d1) at (0,-1) [label=left:$d_1$]{$1$};
      \node (d2) at (2,-1) [label=right:$d_2$]{$1$};
    \foreach \from/\to in {c3/c2}
    \draw (\from) -- (\to);
        \foreach \from/\to in {a1/b1, a1/b2, b1/c1, b2/c2, b2/c3, c1/d1, c2/d2, c3/d2}
 \draw[ultra thick] (\from) -- (\to);
    \end{tikzpicture}
    \hfill
    \begin{tikzpicture}[every node/.style={circle,fill=black!20}, scale=.65]  
      \node (a1) at (-1,2)[label=left:$a_1$]{$4$};
      \node (a2) at (3,2)[label=right:$a_2$]{$4$};
      \node (b1) at (0,1) [label=left:$b_1$]{$3$};
      \node (b2) at (2,1) [label=right:$b_2$]{$3$};
      \node (c1) at (-1,0) [label=left:$c_1$]{$2$};
      \node (c2) at (1,0) [label=left:$c_2$]{$2$};
      \node (c3) at (3,0) [label=right:$c_3$]{$2$};
      \node (d1) at (0,-1) [label=left:$d_1$]{$1$};
      \node (d2) at (2,-1) [label=right:$d_2$]{$1$};
    \foreach \from/\to in {c3/c2}
    \draw (\from) -- (\to);
        \foreach \from/\to in {a1/b1, a2/b2, b1/c1, b2/c2, b2/c3, c1/d1, c2/d2, c3/d2}
 \draw[ultra thick] (\from) -- (\to);
    \end{tikzpicture}
\end{center}

\caption{The unique graphs realizing $(1,3,2)$ and $(1,2,3,2)$, and the four lowest ranks of the unique graph realizing $(1,2,\ldots,2,2,3,2)$.}\label{1232}
\end{figure}

We now consider the case where \graph{G} is a graph of rank at least 5 that realizes $(1,2,\ldots,2,3,2)$. Let $\X_4 = \{a_1, a_2\}$, $\X_3 = \{b_1, b_2\}$, $\X_2 = \{c_1, c_2, c_3\}$ and $\X_1 = \{d_1, d_2\}$; we will show, without loss of generality, that the graph induced by these vertices of rank 4 and less, is pictured in Figure~\ref{1232}, on the right side, and that there are no edges between \uprank{\graph{G}}{5} and the vertices of rank less than 4.
By Lemma~\ref{uoddp}, 
\begin{quote}
$(\bigstar)$  \uprank{\graph{G}}{3} is uniquely realized as a path.
\end{quote}
\noindent
By $(\star)$, and
without loss of generality, 
$a_1$ is adjacent to $b_1$, $a_2$ is adjacent to $b_2$, and the distance between $b_1$ and $b_2$ in \uprank{\graph{G}}{3} is at least 4.  Thus by \pathcon, the distance between $b_1$ and $b_2$ in \graph{G} is at least 4.  Thus $b_1$ and $b_2$ cannot share any neighbors of rank 2, so without loss of generality we can assume $b_1$ is adjacent to $c_1$ but not $c_2$ and $b_2$ is adjacent to $c_2$, but not $c_1$.  We now make an \emph{observation}.
\begin{quote}
\emph{If the only rank 2 neighbor of $b_i$ is $c_i$, then $b_i$ must strictly corner $c_i$ in \uprank{\graph{G}}{2}.} 
\end{quote}

\noindent
Consider why the observation is true.
Since $c_i$ is adjacent to $b_i$, by $(\star)$, the only vertices that could strictly corner $c_i$ in \uprank{\graph{G}}{2} are $a_i$ and $b_i$.  If $a_i$ strictly cornered $c_i$ in \uprank{\graph{G}}{2} then it would also strictly corner $b_i$ in \uprank{\graph{G}}{2}, which cannot happen, so $b_i$ must strictly corner $c_i$ in \uprank{\graph{G}}{2}.  So the observation is true.

As mentioned above, at most one of $b_1$ or $b_2$ can be adjacent to $c_3$, so for some $i$, the only rank 2 neighbor of $b_i$ is $c_i$.  Thus the shortest path in \uprank{\graph{G}}{2} between $c_1$ and $c_2$ must include $b_i$ and $a_i$, so by \pathcon, $c_1$ and $c_2$ cannot be adjacent, nor adjacent to the same vertex.  Thus without loss of generality, $d_1$ is adjacent to $c_1$ and not $c_2$, and $d_2$ is adjacent to $c_2$ and not $c_1$.

Now $c_3$ must be adjacent to one of the rank 1 vertices, without loss of generality $d_2$.
Since $c_2$ and $c_3$ are at distance at most 2 in \graph{G}, by \pathcon, in \uprank{\graph{G}}{2} they are at distance at most 2, from which we can conclude that there is a vertex $x$ in \uprank{\graph{G}}{3} that is adjacent to both $c_2$ and $c_3$ (note that if $c_2$ and $c_3$ were adjacent, then the vertex $x$ will be the vertex that strictly corners $c_3$ in \uprank{\graph{G}}{2}).
We show that $b_2$ must be adjacent to $c_3$, by assuming for contradiction that it were not.  Then by the \emph{observation}, $b_2$ must strictly corner $c_2$ in \uprank{\graph{G}}{2}, so $a_2$ is not adjacent to $c_2$ and so cannot be $x$. By assumption, $x$ is not $b_2$.
Since $b_2$ strictly corners $c_2$ in \uprank{\graph{G}}{2}, $b_2$ has to be adjacent to $x$ violating $(\star)$. 
So we have that $b_2$ is adjacent to both $c_2$ and $c_3$.  Thus, just as we argued that $c_2$ is not adjacent to $d_1$, so $c_3$ is not adjacent to $d_1$.

Now, by $(\star)$, only $a_2$ or $b_2$ can strictly corner either $c_2$ or $c_3$ in \uprank{\graph{G}}{2}, but since $a_2$ cannot be adjacent to both $c_2$ and $c_3$, $b_2$ must strictly corner at least one of $c_2$ and $c_3$; without loss of generality, assume $b_2$ strictly corners $c_3$ in \uprank{\graph{G}}{2}.  
Now consider what vertex $y$ strictly corners $d_2$.  The vertex $y$ would have to be adjacent to at least $d_2, c_2,$ and $c_3$.  We know $y \neq a_2$ since $a_2$ cannot be adjacent to both $c_2$ and $c_3$.  The vertex $y$ cannot be another vertex in  \uprank{\graph{G}}{4}, since then $y$ would be adjacent to $c_3$ and since $b_2$ strictly corners $c_3$ in \uprank{\graph{G}}{2}, $b_2$ would have to be adjacent to $y$, violating $(\star)$.  The vertex $y$ can also not be $b_2$ since then $b_2$ would in fact strictly corner $c_3$ in \graph{G}.  Thus $d_2$ is strictly cornered by one of $c_2$ or $c_3$, meaning that $c_2$ is adjacent to $c_3$.
Viewing Figure~\ref{1232}, we have shown that all the displayed edges must be there and have ruled out most of the missing edges; we just need to rule out a few more edges.   We rule out any other edges attached to $c_2$ by considering what could corner $c_2$ in \uprank{\graph{G}}{2}: not $a_2$ since then $a_2$ would be adjacent to $c_2$ and $c_3$, and not any other vertex in \uprank{\graph{G}}{4}, since by $(\star)$ it is not adjacent to $b_2$.   So only $b_2$ can strictly corner  $c_2$ in \uprank{\graph{G}}{2}, so there can be no more edges attached to $c_2$.
We rule out an edge between $d_1$ and $d_2$ using \pathcon, since by the reasoning to this point we can now conclude that the distance between $c_1$ and $c_2$ is at least 5 in \uprank{\graph{G}}{2}.
Also $d_2$ can have no neighbors besides $c_2$ and $c_3$ because if it did, then nothing could strictly corner it; similarly, $d_1$ can have no other neighbors besides  $c_1$.

{\bf Proof of (iii)}:   
 The argument is the same as the one for $(i)$, with $0$-realizations of $(3,2)$, $(2,3,2)$ and $(2, \ldots, 2, 3,2)$ in place of $(1,3,2)$, $(1,2,3,2)$, and $(1,2, \ldots, 2,3,2)$.


{\bf Proofs of (ii) and (iv)}:
Assume for contradiction that we had a graph \graph{G} realizing the appropriate list.  Thus \uprank{\graph{G}}{2} is as described in parts (i) and (iii), so the two rank 2 vertices of \graph{G} are at distance greater than 2 in \uprank{\graph{G}}{2}, but by Corollary~\ref{xk1}, must both be adjacent to the rank 1 vertex in \graph{G}, contradicting \pathcon.
\end{proof}

\begin{figure}
\begin{center}
  \begin{tikzpicture}[every node/.style={circle,fill=black!20}]  
      \node (v) at (0,4)  {$4$};
      \node (d) at (2,3)   {$3$};
      \node (c) at (1,2)    {$3$};
      \node (b) at (-1,2)   {$3$};
      \node (a) at (-2,3)   {$3$};
      \node (y) at (1,1)    {$2$};
      \node (x) at (-1,1)   {$2$};
      \node (w) at (0,0)    {$1$};
    \foreach \from/\to in {x/y, a/b, b/c, c/d, d/a, x/a, x/c, y/b, y/d}
    \draw (\from) -- (\to);
        \foreach \from/\to in {w/x, w/y, a/v, b/v, c/v, d/v, x/b, y/c}
    \draw[ultra thick] (\from) -- (\to);
    \end{tikzpicture}
 \hspace{.3in}
      \begin{tikzpicture}[every node/.style={circle,fill=black!20}, scale=1.2]  
      \node (a) at (0,4) {$4$};
      \node (b) at (2,4.5)  {$4$};
      \node (c) at (4,4) {$4$};
      \node (d) at (0,3)  {$3$};
      \node (e) at (2,3)  {$3$};
      \node (f) at (3.5,3){$3$};
      \node (g) at (2,2)  {$2$};
      \node (h) at (4,2)  {$2$};
      \node (i) at (2,1)  {$1$};
    \foreach \from/\to in {a/b, a/c, b/c, a/d, a/e, a/g, b/e, b/f, c/d, c/f, c/h, c/h, e/f, d/i, d/g, d/h, e/g, f/h, g/i, h/i}
    \draw (\from) -- (\to);
    \foreach \from/\to in {a/d, c/d, b/e, b/f, c/h, a/g, d/i}
    \draw[ultra thick] (\from) -- (\to);
    \end{tikzpicture}

\end{center}

\caption{A graph $1$-realizing (1,4,2,1) and a graph $0$-realizing $(3,3,2,1)$.} \label{2531}
\end{figure}
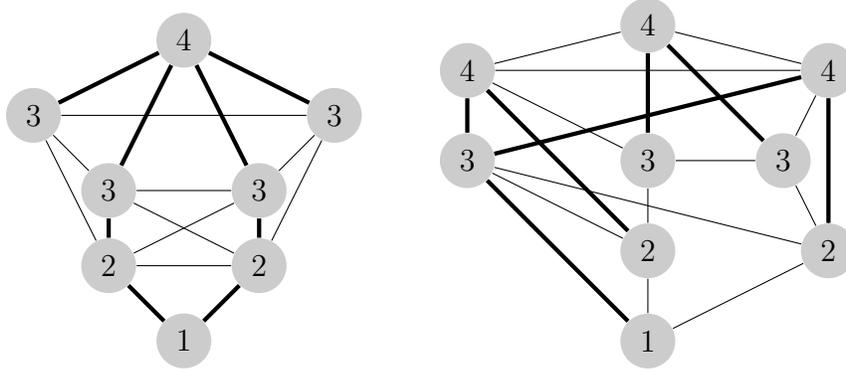

\begin{theorem} \label{thm_n-5}
A cop-win graph on $n \ge 11$ vertices has capture time $n-5$ if and only if
one of the following conditions holds:

\begin{enumerate}
\item It $1$-realizes a standard extension of $(1,4,2,1)$.
\item \label{itemAugment2221}
It $1$-realizes a list formed by taking a standard extension of $(2,2,2,1)$ and then augmenting by adding 1 to any single entry.
\item \label{cond1241} It $0$-realizes a standard extension of $(3,3,2,1)$.
\end{enumerate}
\end{theorem}

\begin{proof}

By Theorem~\ref{thm_rank_capt_time} we know that any graph satisfying one of the conditions does have capture time $n-5$. Observing Figures~\ref{h7} and \ref{2531} we see that we can $1$-realize $(1,4,2,1)$ and $(2,2,2,1)$, and
$0$-realize $(3,3,2,1)$; thus the three classes of graphs in the statement of the theorem are non-empty. 
It  remains to show that our three conditions have not missed any graphs.
Let \graph{G} be a cop-win graph on $n \ge 11$ vertices, with capture time $n-5$, 
with rank cardinality list $\vect{x} = (x_{\alpha}, \ldots, x_1)$.
Since $n \ge 11$, \vect{x} must have length at least 6, and at least one of the first 6 entries of \vect{x}, besides $x_\alpha$, must be a 1 (since otherwise Theorem~\ref{thm_rank_capt_time} would imply \graph{G} has capture time less than $n-5$).
 So suppose $x_i = 1$ and $x_j > 1$ for $i < j < \alpha$, and note that $i \le  \alpha - 3$ by Corollary~\ref{norca-1} and Proposition~\ref{norca-2}.
Consider cases on whether \graph{G} is \tp{0} or \tp{1}.

\begin{itemize}

\item
\emph{Case: \graph{G} is \tp{0}.} 
  
If $x_i$ is $x_{\alpha - 5}$, then in order to have capture time $n-5$, we must have $(2,2,2,2,2,1)$ as an initial segment of \vect{x}, but this list is not $0$-realizable by Lemma~\ref{uoddp}.

If $x_i$ is $x_{\alpha - 4}$, then in order to have capture time $n-5$, we have the following possible initial segments of \vect{x}: $(3,2,2,2,1)$, $(2,3,2,2,1)$, $(2,2,3,2,1)$, or $(2,2,2,3,1)$.
The first three lists are not $0$-realizable by  Lemma~\ref{pathlike}. We can show the list $(2,2,2,3,1)$ is not $0$-realizable using Lemma~\ref{uoddp} and \pathcon.

If $x_i$ is $x_{\alpha - 3}$,  then in order to have capture time $n-5$, the possible initial segments are: $(3,3,2,1)$,  $(3,2,3,1)$,  $(2,3,3,1)$,  $(4,2,2,1)$,  $(2,4,2,1)$, or  $(2,2,4,1)$.
The first list $(3,3,2,1)$ is $0$-realizable as required, but the rest are not $0$-realizable.
The lists $(2,3,3,1)$, $(2,4,2,1)$, $(2,2,4,1)$ are not 0-realizable by Lemma~\ref{2mk1con}, and
the lists $(3,2,3,1)$ and $(4,2,2,1)$ are not $0$-realizable by Lemma~\ref{m2k1}.

\item
\emph{Case: \graph{G} is \tp{1}.} 


If $x_i$ is $x_{\alpha - 5}$, then in order to have capture time $n-5$, we must have $(1,2,2,2,2,1)$ as an initial segment of \vect{x}, but this list is not realizable by Lemma~\ref{uoddp}.

If $x_i$ is $x_{\alpha - 4}$, then in order to have capture time $n-5$, we have the following possible initial segments of \vect{x}: $(2,2,2,2,1)$,
$(1,3,2,2,1)$, $(1,2,3,2,1)$, or $(1,2,2,3,1)$.
The first list satisfies Condition~\ref{itemAugment2221} in the statement of the theorem and realizable as required. 
By Lemma~\ref{pathlike} the second and third lists are not realizable.  
The last list is not realizable by Lemma~\ref{uoddp} and \pathcon.

If $x_i$ is $x_{\alpha - 3}$,  then in order to have capture time $n-5$, the possible initial segments are: $(3,2,2,1)$, $(2,3,2,1)$, $(2,2,3,1)$, $(1,4,2,1)$, $(1,2,4,1)$, or $(1,3,3,1)$.
The first four lists are realizable as required (The first three satisfy Condition~\ref{itemAugment2221}, and the fourth satisfies Condition~\ref{cond1241}).  
The last two lists are not realizable by Corollary~\ref{cor_not12k1_not13k1}.

\end{itemize}
\end{proof}

\section{Future Work}\label{future}

The main results of our paper are structural characterizations of $\mathcal{G}_n^{n-4}$ and $\mathcal{G}_n^{n-5}$, which suggests the following open question.

\begin{question}
Find structural characterizations of $\mathcal{G}_n^s$ for all $s \le n-4$.
\end{question}

\noindent
Our approach is to give the charaterization in terms of what lists the graphs should realize.  With some terminology, we will be more specific about our approach.

\begin{defn} A list \vect{x}, of length at least 2, is \dword{$t$-minimal} if 
the only $t$-realizable list $\le \vect{x}$, of length at least 2, is \vect{x} itself. 
A list is \dword{minimal} if it is either $0$-minimal or $1$-minimal. 
\end{defn}

\noindent
For example, it follows from Theorem~\ref{gave} that $(2,2,2,1)$ is $1$-minimal, and thus, for example,
$(2,7,2,1)$ and $(2,2,2,1,1,1)$ are not $1$-minimal.  We can restate the crux of our main results (recall the ordering on lists from Definition~\ref{def_vecop}).
Theorem~\ref{gave} states that for $n \ge 9$, $\mathcal{G}_n^{n-4}$ is the set of graphs with $n$ vertices that $1$-realize a list of length $n-3$ which is larger than $(2,2,2,1)$.
Theorem~\ref{thm_n-5} states that for $n \ge 11$, $\mathcal{G}_n^{n-5}$ is the set of graphs with $n$ vertices that either:
1) $0$-realize a list of length $n-5$ which is larger than $(3,3,2,1)$, or
2) $1$-realize a list of length $n-4$ which is larger than $(1,4,2,1)$ or $(2,2,2,1)$.

A general approach to characterizing some $\mathcal{G}_n^s$ is to find the appropriate minimal lists and take the appropriate length lists that are larger.  The key technical point then becomes determining which lists are minimal.  In other words, we can make Question~\ref{mainq} more specific:

\begin{question} \label{last_quest}
For $t \in \{0, 1\}$, which lists are $t$-minimal?
\end{question}
From the results in this paper we can conclude that the lists
$(1,2)$, $(1,4,2,1)$, and $(2,2,2,1)$ are $1$-minimal, and 
the lists $(2,2)$,  $(2,5,3,1)$,  $(2,6,2,1)$, and $(3,3,2,1)$ are $0$-minimal.  In our unpublished note \cite{OO17_Ex}, 
we have more examples and speculations relating to Question~\ref{last_quest}.

\appendix

\section*{Appendix} 

The results proved in this appendix are included to keep the paper self-contained.  However more general results, that include those here, are contained in our submitted paper \cite{OO16} which was unpublished at the time this paper was published.
Similar results are proved in \cite{CFM2014}; the relationship to our approach is discussed in \cite{OO16}. 

\vspace{2mm}
\noindent
The proof that the corner ranking procedure is well defined follows.
\begin{proof}
To show that the corner ranking procedure is well-defined on cop-win graphs, it suffices to show if \graph{G} is a cop-win graph which is not a clique, then \graph{G} must have a strict corner.  Supposing \graph{G} is a non-clique which has no strict corners, we show that it is not cop-win.  Let $v \in V(\graph{G})$ be some corner, and let $v_1, \ldots, v_k$ be all the vertices which corner $v$, which means that any two vertices  among $v, v_1, \ldots, v_k$ are twins.  Using the idea of corner elimination, $\graph{G}$ is cop-win if and only if the graph $\graph{G}'$ obtained by deleting the corners $v_1, \ldots, v_k$ is cop-win. Note that $v$ is not a corner in $\graph{G}'$, and since \graph{G} was not a clique, $\graph{G}'$ 
is not a clique.
If $\graph{G}'$ has no corners, then \graph{G} is not cop-win.  Otherwise, repeating the removal process for the remaining sets of twins in $\graph{G}'$ will eventually result in a graph that has no corners.  Thus \graph{G} is not cop-win.

\end{proof}

\noindent
The proof of Lemma~\ref{higran} follows.
\begin{proof}
Suppose $(v_1, \ldots, v_k)$ is a maximal sequence of strict corners such that $v_1 = v$ and each $v_{i+1}$ strictly corners $v_i$. 
Note that $v_j$ strictly corners $v_i$ if $i < j$, so $v_k$ strictly corners $v$. 
Since $v_k$ is a strict corner, it must be strictly cornered by some vertex $w \notin \{v_1, \ldots, v_k \}$.  By the maximality of the sequence, $w$ is not a strict corner.
Since $v$ is strictly cornered by $v_k$ and $v_k$ is strictly cornered by $w$, $v$ is strictly cornered by $w$, which is not a strict corner, and thus $w$ is of higher rank.
\end{proof}

\noindent
To prove Theorem~\ref{thm_rank_capt_time}, we need to define projection functions relative to corner rank. For any graph 
\graph{G}, define $\power{\graph{G}}$ to be the non-empty subsets of $V(\graph{G})$.

\begin{def_nonum} 
Suppose \graph{G} is a graph with corner rank $\alpha$.
We define the functions $f_1, \ldots, f_{\alpha -1}$ and 
$F_1, \ldots, F_{\alpha-1}, F_{\alpha}$,
where
$f_k: \power{\uprank{\graph{G}}{k}}  \to \power{ \uprank{\graph{G}}{k+1}}$ and
 $F_k:\power{\graph{G}} \to \power{\uprank{\graph{G}}{k}}$.

\begin{itemize}

\item
For a single vertex $u \in V(\uprank{\graph{G}}{k})$, define
$$f_k(\{u\}) = 
\begin{cases}
\{u\} & \hbox{if $\rfunct{u} > k$} \\
\hbox{the set of vertices in \uprank{\graph{G}}{k+1} 
that strictly corner $u$ in \uprank{\graph{G}}{k} } & \hbox{if $\rfunct{u} = k$.}
\end{cases}
$$

\item $f_k(\{u_1, \ldots, u_t\}) = \bigcup\limits_{1\le i \le t}f_k(\{ u_i \})$.

\item
Let $F_1:  \power{\graph{G}} \to \power{\graph{G}}$ be the identity function.

\item
For $k \ge 2$, let $F_k = f_{k-1} \circ \cdots \circ f_1$.

\end{itemize}

\end{def_nonum}
\noindent
For a function $h$ whose domain is sets of vertices, we adopt the usual convention that $h(u) = h(\{ u \})$ for a single vertex $u$. We remark that by
by Lemma~\ref{higran} the functions $f_k$ are guaranteed to have
non-empty sets for values.
We say $v$ is a \dword{$k$-projection} (or simply a \dword{projection}) of $w$ if $v \in F_k(w)$.

\begin{def_nonum}
Let \graph{H} and \graph{G} be two graphs. We say the function $h: \power{\graph{H}} \to \power{\graph{G}}$ is a \dword{homomorphism} if 
all the vertices of $h(U)$ are adjacent to all the vertices of $h(V)$ whenever all the 
vertices of $U$ are adjacent to all the vertices of $V$.


\end{def_nonum}

\begin{lem_nonum}  
For any graph with corner rank $\alpha$, its associated functions
$f_1, \ldots, f_{\alpha -1}$ and 
$F_1, \ldots, F_{\alpha-1}, F_{\alpha}$ are homomorphisms.
\end{lem_nonum}

\begin{proof}
Let \graph{G} be the graph.
The identity function $F_1$ is a homomorphism. We show that each $f_k$ is a homomorphism, which implies all other $F_k$'s are homomorphisms
because the property is preserved by composition.  
We prove that $f_k$ is a homomorphism when the sets $U$ and $V$ consist of just the single vertices $u$ and $v$, respectively.
The general case then follows immediately.
Suppose $u, v \in V(\uprank{\graph{G}}{k})$ are distinct and adjacent, and let $u^* \in f_k(u)$ and
$v^* \in f_k(v)$; we show that $u^*$ and $v^*$ are adjacent.  Note that even if $u^* = v^*$, the argument works since our graphs are reflexive.

\textit{Case 1:}  Suppose $u,v \in V(\uprank{\graph{G}}{k+1})$.  Then $f_k(v) = \{ v \}$ and $f_k(u) = \{ u \}$,
and so $u^* = u$ and $v^* = v$ are  adjacent.

\textit{Case 2:} Suppose  $v \in V(\uprank{\graph{G}}{k+1})$, $u \notin V(\uprank{\graph{G}}{k+1})$.
So $v^* = v$ and $u^* \neq u$.  
Since $u^*$ strictly corners $u$ in \uprank{\graph{G}}{k}, $u^*$ is adjacent to $v$, and thus $u^*$  and $v^*$ are adjacent.

\textit{Case 3:} Suppose $u,v \notin V(\uprank{\graph{G}}{k+1})$. 
So $v^* \neq v$ and $u^* \neq u$.  
Since $u^*$ strictly corners $u$ in \uprank{\graph{G}}{k}, $u^*$ is adjacent to both $u$ and $v$. Since $v^*$ strictly corners $v$ in
\uprank{\graph{G}}{k}, and $v$ is adjacent to $u^*$, we have that
$v^*$ is also adjacent to $u^*$. 
\end{proof}

\noindent
The proof of Theorem~\ref{thm_rank_capt_time} follows.
\begin{proof}
Let \graph{G} be a \tp{t} cop-win graph with corner rank $\alpha$.  To show that $\capt({\graph{G}}) = \alpha - t$, we prove an upper and lower bound on $\capt({\graph{G}})$.

First we show $\capt({\graph{G}}) \le \alpha - t$.
For $k \ge 1$, we say that the robber is \dword{$k$-caught} if the cop is at some vertex $c$, the robber at some vertex $r$, and $c \in F_k(r)$.
We describe a strategy on \graph{G} for the cop that succeeds in at most $\alpha - t$ cop moves.  The cop starts at any vertex of corner rank $\alpha$.  
No matter where the robber starts, if
\graph{G} is \tp{1}, then since the cop dominates the top two ranks of vertices, the cop can play so after one cop move, the robber is $(\alpha - 1)$-caught.  Similarly,  if \graph{G} is \tp{0}, then the cop can  play so that after one move the robber is $\alpha$-caught.
We show the following claim:
\begin{quote}
If the robber is $k$-caught, for $k \ge 2$, then 
then for any robber move, there is a cop move which leaves the robber
$(k-1)$-caught.
\end{quote}
Proving the claim proves the upper bound since we can repeatedly apply the claim and once the robber is $1$-caught, the robber is actually caught.  Now we prove the claim, where we suppose the cop is at $c$ and the robber is at $r$.
Since $c \in F_k(r) =  f_{k-1} \circ  F_{k-1} (r)$, either 
$c \in F_{k-1}(r)$ or 
there is an $r' \in F_{k-1}(r)$ such that $c$ strictly corners $r'$ in 
\uprank{\graph{G}}{k-1}.  Either way, there is a $(k-1)$-projection $r'$ of $r$ such that $c$ corners $r'$ in \uprank{\graph{G}}{k-1}.  Thus, since $F_{k-1}$ is a homomorphism, wherever the robber moves to, from $r$, the cop can move so that the robber is $(k-1)$-caught.

Now we show $\capt({\graph{G}}) \ge \alpha - t$.
We say that a robber location at vertex $r$ is $k$-\dword{proj-safe} if its corner rank is at least $k$ and the cop is at a vertex $c$ such that
there is some $c' \in F_k(c)$ such that $c'$ is not adjacent to $r$.
Since $F_k$ is a homomorphism, if a location is $k$-proj-safe, then the cop is not adjacent to the robber.

If \graph{G} is \tp{0}, then the robber starts at a vertex which is $(\alpha - 1)$-proj-safe, while if \graph{G} is 
\tp{1}, then the robber starts at a vertex which is $(\alpha - 2)$-proj-safe. We show that starting in such a way is possible.
We use the fact that since $F_k$ is a homomorphism, the vertices of rank $k$ or higher adjacent to a vertex $c$ are a subset of the vertices of rank $k$ or higher adjacent to a vertex from $F_k(c)$.
Thus to see that a $(\alpha - 1)$-proj-safe start is possible in the \tp{0} case, it suffices to show there is no vertex in \uprank{\graph{G}}{\alpha - 1} that dominates  \uprank{\graph{G}}{\alpha - 1}.  If there were such a vertex, it would have to be rank $\alpha$, but then the graph would be \tp{1}. Similarly, to see that such a start is possible in the 
\tp{1} case, 
it suffices to show that there is no vertex in
\uprank{\graph{G}}{\alpha - 2} that dominates \uprank{\graph{G}}{\alpha - 2}.   Suppose for the sake of contradiction there is such a vertex $v$. 
In 
\uprank{\graph{G}}{\alpha - 2}, $v$ cannot strictly corner any vertex in 
\uprank{\graph{G}}{\alpha - 1}, so all the vertices of 
\uprank{\graph{G}}{\alpha - 1}
are twins with $v$.  But the fact that a vertex of rank $\alpha - 1$ is twins with a vertex of rank $\alpha$ leads to a contradiction.
The lower bound will follow once we prove the following claim:
\begin{quote}
If the robber location is $k$-proj-safe, for $k \ge 2$, then no matter what the cop does, the robber has a move to a $(k-1)$-proj-safe location.
\end{quote}
To prove the claim, suppose the robber is at the $k$-proj-safe vertex $r_0$, and the cop is at $c_0$.  Thus there exists $c'_0 \in F_k(c_0)$ such that $c'_0$ is not adjacent to $r_0$.  Suppose the cop then moves to $c_1$.  
Assume for the sake of contradiction that from $r_0$ the robber does not have a move to a $(k-1)$-proj-safe vertex.
For $r_0$ not to have such a move, means 
that all $c_1' \in F_{k-1}(c_1)$ corner $r_0$ in 
\uprank{\graph{G}}{k-1}.
Consider one such $c_1'$. 
Since $r_0$ has rank at least $k$, and thus cannot be a strict corner in \uprank{\graph{G}}{k-1}, this cornering cannot be strict, and thus $c_1'$ and $r_0$ are twins in 
\uprank{\graph{G}}{k-1}. Since $r_0$ has rank at least $k$, $c_1'$ must also have rank at least $k$.  This implies that $f_{k-1}(c_1') = \{c_1'\}$ and so $c_1' \in F_k(c_1) =  f_{k-1} \circ  F_{k-1} (c_1)$. However since $F_k$ is a homomorphism and $c_1$ is adjacent to $c_0$, $c_1'$ is adjacent to $c_0'$. Since $r_0$ is not adjacent to $c_0'$, this contradicts the fact that $c_1'$ and $r_0$ are twins in 
\uprank{\graph{G}}{k-1}.
\end{proof}

\end{document}